\documentclass[11pt,english]{smfart}

\usepackage[T2A,T1]{fontenc}
\usepackage[english,francais]{babel}
\usepackage[utf8]{inputenc}

\usepackage{latexsym,amscd,color}
\usepackage{amsmath,amsfonts,amssymb,mathrsfs,mathtools}
\usepackage{bm,pifont}
\Pifont{pzd}
\usepackage{enumitem,euscript}
 \usepackage{soul}
\makeatletter
\def\namedlabel#1#2{\begingroup
    #2%
    \def\@currentlabel{#2}%
    \phantomsection\label{#1}\endgroup
}
\makeatother

\usepackage{amssymb,url,xspace,smfthm,hyperref,CJK}
\usepackage{comment}
\usepackage{cancel}

\input xy
\xyoption{all}

\numberwithin{equation}{section}

\newcommand{\AMS}{$\mathcal{A}$\mathrm{Ker}n-.1667em\lower.5ex\hbox
        {$\mathcal{M}$}\mathrm{Ker}n-.125em$\mathcal{S}$}

\DeclareMathOperator{\Ass}{Ass}
\DeclareMathOperator{\Ad}{Ad}

\DeclareMathOperator{\fppf}{fppf}

\DeclareMathOperator{\Lie}{Lie}

\DeclareMathOperator{\rang}{rk}

\DeclareMathOperator{\Spec}{Spec}
\DeclareMathOperator{\Supp}{Supp}

\newcommand{\emptyinnprod}{\langle\kern.3em,\kern.13em\rangle}

\def\esssup_#1{\underset{#1}{\mathrm{ess\,sup\, }}}
\newcommand{\ndot}{\raisebox{.4ex}{.}}

\def\ifquery{\if00}
\ifquery
\def\query#1{\setlength\marginparwidth{65pt} 
\marginpar{\raggedright\fontsize{7.81}{9} 
\selectfont\upshape\hrule\smallskip 
#1\par\smallskip\hrule}} 

\else
\def\query#1{}
\fi

\definecolor{Huayi}{rgb}{0.0, 0.0, 1.0}

\definecolor{Hide}{rgb}{1.0, 0.0, 0.0}

\title{Harder-Narasimhan games}
\date{\today}
\author{Huayi Chen}
\address{Universit\'e de Paris and Sorbonne Universit\'e, CNRS, IMJ-PRG, F-75013 Paris, France}
\email{huayi.chen@imj-prg.fr}
\urladdr{webusers.imj-prg.fr/~huayi.chen}

\author{Marion Jeannin}
\address{Uppsala University}
\email{marion.jeannin@math.uu.se}
\urladdr{https://www.katalog.uu.se/empinfo/?id=N22-369}

\begin{document}
\def\smfbyname{}

\begin{abstract}
In this article, we study the notion of semi-stability and the Harder-Narasimhan filtration from a game-theoretic point of view. This allows us to provide a unified proof for the existence and uniqueness of the Harder-Narasimhan filtration in various settings and offer a conceptual interpretation for semi-stability conditions. As an application, we establish the existence and uniqueness of a coprimary filtration for modules of finite type over a commutative N\oe therian ring and interpret it as a Harder-Narasimhan filtration.  
\end{abstract}


\maketitle

\tableofcontents



\section{Introduction}

Harder-Narasimhan filtration is a classic construction in algebraic geometry. It originates from the notion of stability  introduced by Mumford \cite{MR0175899} and Takemoto \cite{MR337966} in the study of moduli spaces of vector bundles. Let $C$ be a regular projective curve over a field $k$. For any non-zero vector bundle $E$ on $C$, the \emph{slope} of $E$ is defined as the quotient of $\deg(E)$ by the rank of $E$ and is denoted by $\mu(E)$. The vector bundle $E$ is said to be \emph{semi-stable} if the slope of any non-zero vector subbundle of $E$ is bounded from above by $\mu(E)$. Harder and Narasimhan \cite{MR364254} have shown that any non-zero vector bundle $E$ has a unique filtration by vector subbundles
\[0=E_0\subsetneq E_1\subsetneq \ldots\subsetneq E_n=E\]
such that all subquotient sheaves $E_i/E_{i-1}$ are semi-stable vector bundles over $C$ and that the successive slopes satisfy the following inequalities:
\[\mu(E_1/E_0)>\ldots>\mu(E_n/E_{n-1}).\]
This result has then been generalized to higher-dimensional projective varieties by Shatz \cite{MR498573} and Maruyama \cite{MR615853}. Moreover, in the context of  complex analytic geometry, Bruasse \cite{MR1843867,MR1990006} has established an analogue of Harder-Narasimhan filtration for Hermitian vector bundles equipped with a semi-connection.

The notion of semi-stability and the  construction of Harder-Narasimhan filtration have analogues in various branches of mathematics, which  often exhibit quite different natures. Examples include quiver representations  \cite{MR1906875} in representation theory,  flat vector bundles on an affine manifold \cite{MR2903192} in differential geometry, quasi-coherent sheaves of finite rank on a non-commutative torus  \cite{MR2330589} in non-commutative geometry, finite and flat group schemes over a valuation ring of mixed characteristic \cite{MR2673421} in arithmetic geometry, filtered isocrystals \cite{MR2119719} in $p$-adic Hodge theory, and linear codes \cite{MR3912956} in the theory of error correction codes.

Motivated by the similarity of Harder-Narasimhan filtrations in various contexts, several formalisms  have been proposed in the literature, either from a categorical point of view (see for example \cite{MR1480783,MR2373143,MR2559690,MR2571693}), or from an order theory point of view \cite{MR3928223}. Most of these works begin with a small category $\mathcal C$ with a  zero object, on which Grothendieck's group is defined. Given two morphisms  $\deg(\ndot)$ and $\operatorname{rk}(\ndot)$ from the Grothendieck group of $\mathcal C$ to $\mathbb R$, such that $\operatorname{rk}(\ndot)$ takes non-zero values on classes of non-zero objects, one defines the slope function $\mu(\ndot)$ on the set of non-zero objects as  \[\big(X\in\operatorname{obj}(\mathcal C)\big)\longmapsto \frac{\deg(X)}{\operatorname{rk}(X)},\] where, by abuse of notation we still denote by $X$ the class in the Grothendieck group that it represents, see for example \cite[\S II.2.a)]{MR1600006}, 
\cite[\S1.2]{MR2373143}, \cite[\S1.3.1]{MR2571693}, \cite[\S4.1]{MR2559690}. In the approach of modular lattice, the slope function is also defined in a similar way, see \cite[\S2.4]{MR3928223}.  

In the work \cite{MR4292529}, an analogue of Harder-Narasimhan theory has been proposed in the framework of adelic vector bundles, which is not included in the formalisms mentioned above. For convenience of readers, we recall this construction in the particular case of normed lattices. By \emph{normed lattice}, we mean a free abelian group of finite type $E$, equipped with a norm $\|\ndot\|$ on the finite-dimensional real vector space $E_{\mathbb R}= E\otimes_{\mathbb Z}\mathbb R$. We often use the expression $\overline{E}$ to denote the pair $(E,\|\ndot\|)$. The \emph{rank} of the normed lattice $(E,\|\ndot\|)$ is defined as that of $E$ over $\mathbb Z$, which also identifies with the dimension of the real vector space $E_{\mathbb R}$. Clearly, a subgroup $F$ of $E$ equipped with the restriction of $\|\ndot\|$ to $F_{\mathbb R}$ also forms a normed lattice. Similarly, a free quotient group $G$ of $E$ equipped with quotient norm of $\|\ndot\|$ on $G_{\mathbb R}$ also forms a normed lattice.

Given a normed lattice $(E,\|\ndot\|)$ of rank $r$, we define the \emph{Arakelov degree} of $(E,\|\ndot\|)$ as 
\[\widehat{\deg}(E,\|\ndot\|)=-\ln\|e_1\wedge\cdots\wedge e_r\|_{\det},\]
where $(e_i)_{i=1}^r$ is a basis of $E$ over $\mathbb Z$, and $\|\ndot\|_{\det}$ is the \emph{determinant norm} associated with $\|\ndot\|$, which is the norm on $\det(E_{\mathbb R})$ defined as 
\[\forall\,\eta\in\det(E_{\mathbb R}),\quad \|\eta\|_{\det}=\inf_{\begin{subarray}{c}(x_i)_{i=1}^r\in E^r\\
\eta=x_1\wedge\cdots\wedge x_r\end{subarray}}\|x_1\|\cdots\|x_r\|.\]
In the case where $E$ is non-zero, the \emph{slope} of $(E,\|\ndot\|)$ is defined as the ratio
\[\widehat{\mu}(E,\|\ndot\|):=\frac{\widehat{\deg}(E,\|\ndot\|)}{\rang_{\mathbb Z}(E)},\] and the minimal slope of $(E,\|\ndot\|)$ is defined as
\[\widehat{\mu}_{\min}(E,\|\ndot\|):=\inf_{E\twoheadrightarrow G\neq\boldsymbol{0}}\widehat{\mu}(\overline G),\]
where $\overline G$ runs over the set of free quotient groups of $E$ equipped with quotient norms. It has been shown in \cite[Theorem 4.3.58]{MR4292529} that, for any non-zero normed lattice $(E,\|\ndot\|)$, there exists a unique filtration 
\[0=E_0\subsetneq E_1\ldots\subsetneq E_n=E\]
of subgroups of $E$ such that each subquotient $E_i/E_{i-1}$ is a free Abelian group and forms a semi-stable normed lattice if we equip it with the subquotient norm, and that the following inequalities are satisfied:
\[\widehat{\mu}_{\min}(\overline{E_1/E_0})>\ldots>\widehat{\mu}_{\min}(\overline{E_n/E_{n-1}}).\] 

In the case where the norm $\|\ndot\|$ is induced by an inner product, this result is due to Stuhler \cite{MR0424707} (see also \cite{MR780079} for the link with reduction theory) and is quite similar with the classic Harder-Narasimhan theory of vector bundles. However, in the case of a general norm, the existence and uniqueness of a flag with semi-stable subquotient and decreasing minimal slopes is rather unexpected. In fact, the Arakelov degree function $\widehat{\deg}(\ndot)$ here is not additive with respect to short exact sequences, which goes beyond the existing Harder-Narasimhan theory. Moreover, the semi-stability condition is also slightly different from the classic formulation:  a non-zero normed lattice $\overline F$ is said to be \emph{semi-stable} if for any non-zero subgroup $F'$ of $F$ one has 
\[\widehat{\mu}_{\min}(\overline{F'})\leqslant\widehat{\mu}_{\min}(\overline F).\] 
Note that, in the case where the norm of $\overline F$ is induced by an inner product, this condition is actually equivalent to 
\[\forall\,F'\subseteq F,\;F'\neq\boldsymbol{0},\quad \widehat{\mu}(\overline{F'})\leqslant\widehat{\mu}(\overline F), \]
which is similar to the semi-stability condition of vector bundles. However, in general, these two conditions are not equivalent and the latter is not adequate for establishing the Harder-Narasimhan theory.

The result \cite[Theorem 4.3.58]{MR4292529} was obtained via an interpretation using $\mathbb R$-filtrations of Harder-Narasimhan theory developed in \cite{MR2768967}. Not only the semi-stability condition is mysterious, but also the proof is rather different from the classic ones. The first attempt to include this result in a categorical formalism is due to Li \cite{Li}, where a Harder-Narasimhan theory is developed in the framework of a  proto-abelian category. The new feature of \cite{Li} is to consider slope functions valued in a totally ordered set, while most formalisms require the slope function to be real-valued and often the quotient of two additive functions. However, the conditions imposed to the slope function in  \cite{Li} seem to be satisfied rather by the maximal slope function in most situations, and the category considered should satisfy strong chain conditions. Therefore, the application of this theory to diverse concrete cases demands an extra effort of adaptation.

Let us go back to the semi-stability condition in the above example. Let $\overline E$ be a non-zero normed lattice. Its semi-stability condition could be interpreted as an equality 
\[\sup_{F\subseteq E}\inf_{F'\subsetneq F}\widehat{\mu}(\overline{F/F'})=\widehat{\mu}_{\min}(\overline E),\]
where $F$ runs over the set of all subgroups of $E$, and $F'$ runs over the set of strict subgroups of $F$ such that $F/F'$ is free.
Moreover, a key argument in the proof of \cite[Theorem 4.3.58]{MR4292529} shows that, for any non-zero normed lattice $\overline E$, the following equality holds (see \cite[\S4.3.7]{MR4292529}):
\[\widehat{\mu}_{\min}(\overline E)=\inf_{F'\subsetneq E}\sup_{F\subseteq E,\, F\supsetneq F'}\widehat{\mu}(\overline{F/F'}),\]
where $F'$ runs over the set of strict subgroups of $E$, and $F$ runs over the set of subgroups of $E$ that contain $F'$ strictly and such that $F/F'$ is free. If we combine this equality with the above interpretation of the semi-stability condition, we can  write the latter as an equality between minimax and maximin:
\[\sup_{ F\subseteq E}\inf_{F'\subsetneq F}\widehat{\mu}(\overline{F/F'})=\inf_{F'\subsetneq E}\sup_{F\subseteq E,\, F\supsetneq F'}\widehat{\mu}(\overline{F/F'}).\]
This is a typical situation considered in game theory.

Inspired by this observation, we construct a  zero-sum game with two players,  Alice and Bob. They pick successively  (either Alice or Bob goes first) two elements $a$ and $b$ respectively in a partially ordered set $(\mathscr L,\leqslant)$ with supremum and infimum of finite subsets (i.e., a bounded lattice), under the constraint $a<b$. Note that this constraint prevents Alice from taking the greatest element $\top$ and Bob from taking the least element $\perp$. Let \[P_{<}(\mathscr L)=\{(a,b)\in\mathscr L^2\,|\,a<b\}\] be the set of strictly ordered pairs of $\mathscr L$. If $(S,\leqslant )$ is a complete partially ordered set and if $\mu:P_{<}(\mathscr L)\rightarrow S$ is a map, then $\mu$ determines a pay-off function of the above game:  Alice aims to minimize the pay-off $\mu(a,b)$ while Bob aims to maximize it. We call this game a \emph{Harder-Narasimhan game}. 

If Alice goes first and picks $a$, then all possible pay-offs are given by
\[\text{$\mu(a,b)$, where $b\in\mathscr L$, $b>a$,}\]
and hence $\sup_{b>a}\mu(a,b)$ delimits from above the possible pay-offs. Therefore the optimal pay-off threshold when Alice goes first is 
\[\mu_A^*:=\inf_{a\in\mathscr L\setminus\{\top\}}\sup_{b>a}\mu(a,b),\]
where $\top$ denotes the greatest element of $\mathscr L$. We say that the Harder-Narasimhan game is \emph{semi-stable} if, for any $y\in\mathscr L$ which is not the least element of $\mathscr L$, it is not less favorable for Alice when considering the restriction of the Harder-Narasimhan game to the subset of $\mathscr L$ of elements $\leqslant y$.

If $(x,y)$ is a pair in $P_{<}(\mathscr L)$, we denote by $\mathscr L_{[x,y]}$ the subset \[\{z\in\mathscr L\,|\,x\leqslant z\leqslant y\}\]
of $\mathscr L$. Note that $P_{<}(\mathscr L_{[x,y]})$ is a subset of $P_{<}(\mathscr L)$, and the restriction of $\mu$ to $P_{<}(\mathscr L_{[x,y]})$ defines the pay-off function of a Harder-Narasimhan game on $\mathscr L_{[x,y]}$. We denote by $\mu_{[x,y]}$ the restriction of $\mu$ to $P_{<}(\mathscr L_{[x,y]})$ and we denote by $\mu_A(x,y)$ the value $\mu_{[x,y],A}^*$. Then the semi-stability the Harder-Narasimhan game with pay-off function $\mu$ can be expressed as 
\[\forall\,y\in\mathscr L\setminus{\perp},\quad \mu_A(\perp,y)\not>\mu_A(\perp,\top).\]

The main result of the article is as follows (the total order condition for $(S,\leqslant)$ could be relaxed if only the existence part is needed), see Definition \ref{Def: HN filtration} and Theorem \ref{Thm: uniqueness of HN filtration}.

\begin{theo}\label{Thm: main theorem}Let $(\mathscr L,\leqslant)$ be a bounded lattice, $(S,\leqslant)$ be a complete totally ordered set, and $\mu:P_{<}(\mathscr L)\rightarrow S$ be a map.
Assume the following conditions:
\begin{enumerate}[label=\rm(\alph*)]
\item The pay-off function $\mu$ is convex, namely, for any $(x,y)\in\mathscr L^2$ such that $x\not\leqslant y$, one has
$\mu(x\wedge y,x)\leqslant\mu(y,x\vee y)$.

\item\label{Item: ACC} The bounded lattice $(\mathscr L,\leqslant)$ satisfies the ascending chain condition, namely, there does not exist any family $(x_n)_{n\in\mathbb N}$ of elements of $\mathscr L$ such that $x_0<x_1<\ldots<x_n<x_{n+1}<\ldots$.

\item\label{Item: muADCC} The bounded lattice $(\mathscr L,\leqslant)$ satisfies the $\mu_A$-descending chain condition, namely, for any $a\in\mathscr L$, there does not exist any family $(x_n)_{n\in\mathbb N}$ of elements of $\mathscr L$ such that $x_n>a$ for any $n$ and
$x_0>x_1>\ldots>x_n>x_{n+1}>\ldots$, 
that satisfies the following slope inequalities:
\[\mu_A(a,x_0)<\mu_A(a,x_1)<\ldots<\mu_A(a,x_n)<\mu_A(a,x_{n+1})<\ldots.\]
\end{enumerate}
Then there exists a unique increasing sequence
\[\operatorname{\perp}=a_0<a_1<\ldots<a_n=\top\]
such that: 
\begin{enumerate}[label=\rm(\arabic*)]
\item For any $i\in\{1,\ldots,n\}$, the Harder-Narasimhan game on $\mathscr L_{[a_{i-1},a_i]}$ with pay-off function $\mu_{[a_{i-1},a_i]}$ is semi-stable.

\item The following inequalities hold:
\[\mu_A(a_0,a_1)>\ldots>\mu_A(a_{n-1},a_n).\]
\end{enumerate}
\end{theo}

In the example of a non-zero normed lattice $\overline E$, we take $\mathscr L=\mathscr L(E)$ to be the set of all subgroups of $E$, ordered by the relation of inclusion $\subseteq$. The pay-off function \[\mu:P_{<}(\mathscr L(E))\longrightarrow[-\infty,+\infty]\] is defined as follows: if $F'$ and $F$ are two subgroups of $E$ such that $F'\subsetneq F$, in the case where $\rang_{\mathbb Z}(F')=\rang_{\mathbb Z}(F)$, we set $\mu(F',F)=+\infty$; otherwise we define
\[\mu(F',F)=\widehat{\mu}\big(\overline{(F/F')/(F/F')_{\operatorname{tor}}}\big)+\frac{\ln(\operatorname{card}((F/F')_{\operatorname{tor}}))}{\rang_{\mathbb Z}(F/F')}.\]
The function $\mu$ then defines the pay-off function of a Harder-Narasimhan game on $\mathscr L(E)$, which satisfies the conditions of Theorem \ref{Thm: main theorem}.  Moreover, if $(F',F)$ is a pair of subgroups of $E$ such that $F/F'$ is a free group of positive rank, then we have 
\[\mu_A(F',F)=\widehat{\mu}_{\min}(\overline{F/F'}).\]
By applying Theorem \ref{Thm: main theorem}, we recover the existence and uniqueness of Harder-Narasimhan filtration,  as proven in \cite[Theorem 4.3.58]{MR4292529}. 
 
In the above theorem, the conditions \ref{Item: ACC} and \ref{Item: muADCC} can be compared with the finite length condition of modular lattices in \cite{MR3928223,MR3912956}.  Additionally, the work \cite{Li} proposes the ascending chain condition and the (usual) descending  chain condition in a categorical setting. Although the ascending chain condition is natural (one often considers N\oe therian modules), the descending chain condition is rather restrictive in most algebraic geometry applications of the Harder-Narasimhan theory. Extra technical efforts are needed to adapt the formalism to include classic examples. However, condition \ref{Item: muADCC}, which takes into account the pay-off function $\mu$ and is much weaker than the classic descending chain condition, 
can be easily verified in most concrete situations. The flexibility of the theorem allows for its applications in various contexts. Let us mention an application in commutative algebra. 

Consider a commutative unitary N\oe therian ring $R$ and a non-zero $R$-module $M$ of finite type. Let $\mathscr L$ be the set of all sub-$R$-modules of $M$, which is equipped with the order of inclusion. This forms a bounded lattice. We equip the prime spectrum $\operatorname{Spec}(R)$ with an arbitrary total order $\leqslant$ that extends the order of inclusion. Let $S_0$ be the collection of all finite subsets of $\operatorname{Spec}(R)$. We equip it with an arbitrary total order $\leqslant$ that extends the order of inclusion and such that, for $(\mathfrak p,\mathfrak q)\in\operatorname{Spec}(R)^2$, $\{\mathfrak p\}\leqslant\{\mathfrak q\}$ in $S_0$ if and only if $\mathfrak p\subseteq\mathfrak q$ in $\operatorname{Spec}(R)$. Let $(S,\leqslant)$ be the Dedekind-MacNeille completion of $(S_0,\leqslant)$.

For any pair $(N',N)$ of sub-$R$-modules of $M$ such that $N'\subsetneq N$, we define $\mu(N',N)=\operatorname{Ass}(N/N')$ as the set of all associated prime ideals of the quotient module $N/N'$. This map $\mu:P_{<}(\mathscr L)\rightarrow S$ defines the pay-off function of a Harder-Narasimhan game. It is worth noting that $\mu_A(N',N)$ is equal to the least element of $\operatorname{Ass}(N/N')$ (see Proposition \ref{Pro: muA N'N}),  and the Harder-Narasimhan game with pay-off function $\mu$ is semi-stable if and only if $\operatorname{Ass}(M)$ has cardinal $1$, namely $M$ is a coprimary $R$-module. 
Furthermore, the bounded lattice $(\mathscr L,\subseteq)$ satisfies the ascending chain condition and the $\mu_A$-descending chain condition. Hence we obtain the following result.

\begin{theo}\label{Thm: coprimary filtration}
There exists a unique sequence 
\[\boldsymbol{0}=M_0\subsetneq M_1\subsetneq\ldots\subsetneq M_n=M\]
of sub-$R$-modules of $M$, along with a strictly decreasing sequence
\[\mathfrak p_1>\ldots>\mathfrak p_n\]
in $(\operatorname{Spec}(R),\leqslant)$ such that, for any $i\in\{1,\ldots,n\}$, the subquotient $R$-module $M_{i}/M_{i-1}$ is $\mathfrak p_i$-coprimary. 
\end{theo}

Note that this result is related to the coprimary decomposition (or secondary representation) of $R$-modules introduced in \cite{MR0342506,MR327744,MR314822}. We say that $M$ admits a coprimary decomposition if it can be written in the form of a finite sum of coprimary submodules. It can be shown that, if $M$ admits a coprimary decomposition $M=N_1+\cdots+N_p$, then it admits a minimal decomposition, where the $R$-modules $N_1,\ldots,N_p$ are coprimary with respect to distinct prime ideals (attached prime ideals) and none of the summands is redundant. From such a decomposition, one can deduce a coprimary filtration predicted by Theorem \ref{Thm: coprimary filtration} (since $R$ is N\oe therian, it can be shown that attached prime ideals are exactly associated prime ideals, see \cite[Theorem 1]{MR517734}).  It is worth noting that only a restrictive family of modules are coprimary decomposable. It is known for Artinian modules \cite{MR314822} and injective modules \cite{MR414538}. In fact, it has been proven by \cite[Corollary 1.6]{MR1273181} that, in the case where $M$ is reduced, it is coprimary decomposable if and only if $R/\operatorname{ann}(M)$ is an Artinian ring. Therefore, Theorem \ref{Thm: coprimary filtration} provides a good alternative to  coprimary decomposition.

If one applies the theory of Harder-Narasimhan game to the classic case of a non-zero vector bundle $E$ over a regular projective curve $C$, one would take $\mathscr L$ to be the set of all vector subbundles of $E$ and the pay-off function would be defined as follows:
\[\forall\,(F',F)\in P_{\subsetneq}(\mathscr L),\quad \mu(F',F):=\begin{cases}
\frac{\deg(F)-\deg(F')}{\operatorname{rk}(F)-\operatorname{rk}(F')},&\text{if $\operatorname{rk}(F)>\operatorname{rk}(F')$},\\
+\infty,&\text{if $\operatorname{rk}(F)=\operatorname{rk}(F')$.}
\end{cases}\] 
Note that the semi-stability condition of the Harder-Narasimhan game appears different from the slope semi-stability of $E$, which is given by
\[\frac{\deg(F)}{\operatorname{rk}(F)}\leqslant\frac{\deg(E)}{\operatorname{rk}(E)}\]
for any non-zero vector subbundle $F$ of $E$. In our article, we prove that this slope semi-stability condition of $E$ is actually equivalent to the semi-stability condition of the corresponding Harder-Narasimhan game. We introduce a concept  called \emph{slope-like} condition for the pay-off function, which is closely related to the see-saw principle discussed in \cite{MR1480783}. Under this condition, we establish a connection between the semi-stability of the Harder-Narasimhan game to the semi-stability condition in its usual form. Moreover, we also establish a connection between the semi-stability of the game and the Nash equilibrium.  This result demonstrates that the game-theoretical framework is well-suited for understanding the Harder-Narasimhan filtration and  generalizing its construction in various settings.

\begin{theo}
Assume that the following conditions are satisfied:
\begin{enumerate}[label=\rm(\roman*)]
\item  For any ascending chain 
$x_0< x_1<\ldots< x_n< x_{n+1}<\ldots$
of elements of $\mathscr L$, 
there exists $N\in\mathbb N$ such that  
$\mu(x_N,x_{N+1})\leqslant\mu(x_{N},\top)$.

\item 
For any descending chain 
$x_0> x_1>\ldots> x_n>x_{n+1}>\ldots$
of elements of $\mathscr L$, there exists $N\in\mathbb N$ such that
$\mu(\perp,x_N)\leqslant\mu(x_{N+1},x_N)$.

\item The complete lattice is totally ordered, and the pay-off function $\mu$ is slope-like \textup{(}i.e., for $x<y<z$, one has $\mu(x,y)<\mu(x,z)<\mu(y,z)$ or $\mu(x,y)>\mu(x,z)>\mu(y,z)$ or $\mu(x,y)=\mu(x,z)=\mu(y,z)$\textup{)}.
\end{enumerate}
Then the following statements are equivalent:
\begin{enumerate}[label=\rm(\alph*)]
\item $\mu_{\max}(\top):=\displaystyle\sup_{x\in\mathscr L\setminus\{\perp\}}\mu(\perp,x)=\mu(\perp,\top)$.
\item $\mu_{\min}(\top):=\displaystyle\inf_{x\in\mathscr L\setminus\{\top\}}\mu(x,\top)=\mu(\perp,\top)$.
\item $\mu_{\min}(\top)=\mu_{\max}(\top)$.
\item The Harder-Narasimhan game with pay-off function $\mu$ has Nash equilibrium.
\item The Harder-Narasimhan game with pay-off function $\mu$ is semi-stable.
\end{enumerate}
\end{theo}

The article is organised as follows: In the second section, we introduce the notion of Harder-Narasimhan game and prove some basic properties of the  optimal pay-off threshold function. In the third section, we consider the semi-stability condition of Harder-Narasimhan games and prove the existence and uniqueness of Harder-Narasimhan filtration. In the fourth section, we discuss the slope-like pay-off function and establish the link between the semi-stability of the Harder-Narasimhan game and the classic slope semi-stability. Additionally, we  discuss Jordan-Hölder filtrations of a semi-stable Harder-Narasimhan game.   


\section{Harder-Narasimhan games}

In this section, we consider a partially ordered set $(\mathscr L,\leqslant)$, which has a greatest element denoted as $\top$ and a least element denoted as $\perp$, with $\operatorname{\perp}\neq\operatorname{\top}$. In the language of order theory, $(\mathscr L,\leqslant)$ is referred to as a \emph{bounded poset}. We use the notation  $P_{<}(\mathscr L)$ to represent  the set of strictly ordered pairs of elements in $\mathscr L$. This set can be defined as follows:
\[P_{<}(\mathscr L):=\{(x,y)\in\mathscr L^2\,|\,x< y\},\]
where $x<y$ signifies the relation $x\leqslant y$ and $x\neq y$.

\subsection{Description of the game}
\phantomsection\label{Sec: HN game}
Let $(S,\leqslant)$ be a complete lattice, namely a partially ordered where every subset of $S$ admits a supremum (least upper bound) and an infimum (greatest lower bound). The assumption of completeness here is unproblematic since we can consider the Dedekind-MacNeille completion (see \cite[\S11]{MR1501929}). Let $\mu:P_{<}(\mathscr L)\rightarrow S$ be a map. We call \emph{Harder-Narasimhan game} on $(\mathscr L,\leqslant)$ with pay-off function $\mu$ a non-cooperative sequential game of two steps described as follows: 

\begin{quote}\it\hskip\parindent Two players, Alice and Bob, pick successively \textup{(}either Alice or Bob goes first\textup{)} two elements $x_A$ and $x_B$ from the set $\mathscr L$, with the constraint $x_A<x_B$. The objective of Alice is to minimize the pay-off $\mu(x_A,x_B)$ while that  Bob aims to maximize it.
\end{quote} 

If Alice goes first, she selects an element of $\mathscr L\setminus\{\top\}$ (as the constraint prevents her from choosing $\top$). Bob then selects an element from $\mathscr L$ that is strictly greater than Alice's choice. Similarly, if Bob goes first, he chooses an element of $\mathscr L$ other than $\perp$, and then Alice selects an element from $\mathscr L$ that is strictly less than what Bob has chosen. Note that this game is asymmetric: depending on whether Alice or Bob starts, the strategy sets differ after the first step of the game.

\begin{defi}
Let $x$ be an element of $\mathscr L\setminus\{\top\}$. We denote by $V_A(x)$ the set of all possible pay-offs in the game if Alice goes first and picks $x$. It can be defined as:
\[V_A(x)=\{\mu(x,y)\,|\,y\in\mathscr L,\;x<y\}.\]
Moreover, $\sup(V_A(x))$ delimits from above possible pay-offs in the case where Alice goes first and picks $x$. We define 
\[\mu_A^*:=\inf_{x\in\mathscr L\setminus\{\top\}}\sup(V_A(x))=\inf_{x\in\mathscr L\setminus\{\top\}}\sup_{\begin{subarray}{c}y\in\mathscr L\\
x<y\end{subarray}}\mu(x,y).\]
This value represents the optimal pay-off threshold when Alice goes first. 
\end{defi}

\begin{defi}
Let $(x,y)$ be an element of $P_{<}(\mathscr L)$. We denote by $\mathscr L_{[x,y]}$ the interval in $\mathscr L$ delimited by $x$ and $y$, defined as
\[\mathscr L_{[x,y]}:=\{w\in\mathscr L\,|\,x\leqslant w\leqslant y\}.\]
Note that $\mathscr L_{[x,y]}$ forms a bounded poset under the restriction of the order $\leqslant$. Its greatest element is $y$, its least element is $x$. Moreover, $P_{<}(\mathscr L_{[x,y]})$ is a subset of $P_{<}(\mathscr L)$, and the restriction of $\mu$ to $P_{<}(\mathscr L_{[x,y]})$ gives the pay-off function of a Harder-Narasimhan game, denoted by $\mu_{[x,y]}$. This Harder-Narasimhan game is called the \emph{restriction} of the initial game to $\mathscr L_{[x,y]}$.

We denote by $\mu_A(x,y)$ the optimal pay-off threshold for the Harder-Narasimhan game restricted to $\mathscr L_{[x,y]}$ when Alice goes first. By definition, one has 
\begin{equation}\label{Equ: mu A def}\mu_A(x,y):=\mu_{[x,y],A}^*=\inf_{\begin{subarray}{c}a\in\mathscr L\\
x\leqslant a<y
\end{subarray}}\sup_{\begin{subarray}{c}
b\in\mathscr L\\
a<b\leqslant y
\end{subarray}}\mu(a,b).
\end{equation}
In the case where $x=\operatorname{\perp}$, the value of $\mu_A(\perp, y)$ is denoted by $\mu_A(y)$ for simplicity.

To simplify the notation, for any $(a,y)\in P_{<}(\mathscr L)$, we denote  $\mu_{\max}(a,y)$ as the element
\[\sup_{\begin{subarray}{c}
b\in\mathscr L\\
a<b\leqslant y
\end{subarray}}\mu(a,b).\]
With this notation, the optimal pay-off threshold $\mu_A(x,y)$ can also be written as 
\[\mu_A(x,y)=\inf_{\begin{subarray}{c}
a\in\mathscr L\\
x\leqslant a<y
\end{subarray}}\mu_{\max}(a,y).\]
\end{defi}

\subsection{Convexity condition}

In this subsection, we assume that any finite subset of $\mathscr L$ has a supremum   and an infimum in $(\mathscr L,\leqslant)$. In the language of order theory, $(\mathscr L,\leqslant)$ is referred to as a \emph{bounded lattice}. If $x$ and $y$ are two elements of $\mathscr L$, we denote  the supremum of $\{x,y\}$ as $x\vee y$, and  the infimum as $x\wedge y$. We also consider a complete lattice $(S,\leqslant)$ together with a map $\mu:P_{<}(\mathscr L)\rightarrow S$, which enables the study of a Harder-Narasimhan game as described in \S\ref{Sec: HN game}.

\begin{defi}\label{Def: convex pay-off}
We say that the pay-off function $\mu$ is \emph{convex} if, for any $(x,y)\in\mathscr L^2$ such that $x\not\leqslant y$, the following inequality holds: 
\[\mu(x\wedge y,x)\leqslant \mu(y,x\vee y).\]
(Note that the condition $x\not\leqslant y$ implies that $x\wedge y<x$ and $y<x\vee y$). It is worth mentioning that, if $\mu$ is convex, the function $\mu_{[x,y]}$ is also convex for any $(x,y)\in P_{<}(\mathscr L)$. The choice of this terminology is motivated by the observation that a convex function in the usual sense has monotonically increasing difference. The convexity condition for the pay-off function $\mu$ is automatically satisfied  when $(\mathscr L,\leqslant)$ is a totally ordered set.
\end{defi}

\begin{lemm}\phantomsection\label{Lem: majoration mu A star}
Assume that the function $\mu:P_{<}(\mathscr L)\rightarrow S$ is convex. Let $w$ and $x$ be two elements of $\mathscr L\setminus\{\perp\}$ such that $x\not\leqslant w$. For any $(u,t)\in\mathscr L^2$ such that $u\leqslant x\wedge w$ and $x\vee w\leqslant t$, the following inequalities hold:
\begin{equation}\label{Equ: mu A m x upper bound}\mu_A(u,x)\leqslant\mu_{\max}(x\wedge w,x)\leqslant \mu_{\max}(w,t).\end{equation}
Moreover, one has
\begin{equation}\label{Equ: second upper bound}\mu_A(u,x)\leqslant\mu_{A}(w,x\vee w).\end{equation} 
\end{lemm}
\begin{proof} The condition $x\not\leqslant w$ implies that $x\wedge w<x$.
The first inequality of \eqref{Equ: mu A m x upper bound} follows directly from the definition of 
\[\mu_A(u,x) :=\inf_{\begin{subarray}{c}a\in\mathscr L\\
u\leqslant a<x
\end{subarray}}\sup_{\begin{subarray}{c}
b\in\mathscr L\\ a< b\leqslant y\end{subarray}} \mu(a,b).\]

For any $b\in\mathscr L$ such that $x\wedge w<b\leqslant x$, one has 
\[x\wedge w=(x\wedge w)\wedge w\leqslant b\wedge w\leqslant x\wedge w.\]
Therefore,
$b\wedge w=x\wedge w<b$, and hence, $b\not\leqslant w$. Moreover, the assumption $b\leqslant x$ implies $b\vee w\leqslant x\vee w\leqslant t$. By the convexity of $\mu$, we obtain: \[\mu(x\wedge w,b)=\mu(b\wedge w,b)\leqslant\mu(w,b\vee w)\leqslant \mu_{\max}(w,t).\]
Taking the supremum with respect to $b$, we obtain the second inequality of \eqref{Equ: mu A m x upper bound}.

Next, we prove the inequality \eqref{Equ: second upper bound}. Let $y$ be an element of $\mathscr L$ such that $w\leqslant y<x\vee w$. We claim that $x\not\leqslant y$ since otherwise
\[x\vee w\leqslant y\vee w=y<x\vee w,\]
which leads to a contradiction. Moreover, one has 
\[x\vee w\leqslant x\vee y\leqslant x\vee (x\vee w)=x\vee w,\]
so that $x\vee w=x\vee y$. By replacing $w$ with $y$ and $t$ with $x\vee w$ in \eqref{Equ: mu A m x upper bound}, we obtain:
\[\mu_A(u,x)\leqslant\mu_{\max}(y,x\vee w).\]
Taking the infimum with respect to $y$, we obtain \eqref{Equ: second upper bound}.
\end{proof}

\begin{rema}
Note that the function $\mu_{\max}:P_{<}(\mathscr L)\rightarrow S$ is also the pay-off function of a Harder-Narasimhan game. Lemma \ref{Lem: majoration mu A star} implies that, if $\mu$ is convex then so is $\mu_{\max}$. Moreover, for any $(a,y)\in P_{<}(\mathscr L)$, one has 
\[\sup_{\begin{subarray}{c}b\in\mathscr L\\
a<b\leqslant y
\end{subarray}}\mu(a,b)=\sup_{\begin{subarray}{c}b\in\mathscr L\\
a<b\leqslant y
\end{subarray}}\mu_{\max}(a,b).\]
This equality implies that $\mu_{\max,A}(x,y)=\mu_{A}(x,y)$ for any $(x,y)\in P_{<}(\mathscr L)$.
\end{rema}

\begin{prop}\phantomsection\label{Pro: suite exacte et mu A}
Let $x$, $y$ and $z$ be elements of $\mathscr L$ such that $x<y<z$. Then the following inequality holds: \begin{equation}\label{Equ: muA(y,z)geq muA(x,z)}\mu_A(y,z)\geqslant\mu_A(x,z).\end{equation}If the function $\mu$ is convex the following statements hold:
\begin{enumerate}[label=\rm(\alph*)]
\item One has
\[\mu_A(x,z)\geqslant\inf\{\mu_A(x,y),\mu_A(y,z)\}.\]
\item If $\mu_A(x,y)\geqslant\mu_A(y,z)$, then  $\mu_A(y,z)=\mu_A(x,z)$; if $\mu_A(x,y)<\mu_A(y,z)$, then  \[\mu_A(x,y)\leqslant\mu_A(x,z)\leqslant\mu_A(y,z).\] 

\item\label{Item: muAxy and muAyz comparable} If $\mu_A(x,y)$ and $\mu_A(y,z)$ are comparable \textup{(}this is true notably when $(S,\leqslant)$ is totally ordered\textup{)}, or if the infimum
\[\mu_A(x,z)=\inf_{\begin{subarray}{c}a\in\mathscr L\\
x\leqslant a<z
\end{subarray}}\mu_{\max}(a,z) \]
is attained, then, \[\text{either $\mu_A(y,z)=\mu_A(x,z)$, or $\mu_A(x,y)\leqslant\mu_A(x,z)<\mu_A(y,z)$.}\]
\end{enumerate} 
\end{prop}
\begin{proof}
By definition one has 
\[\mu_A(y,z):=\inf_{\begin{subarray}{c}a\in\mathscr L\\
y\leqslant a<z
\end{subarray}}\mu_{\max}(a,z)\geqslant\inf_{\begin{subarray}{c}a\in\mathscr L\\
x\leqslant a<z
\end{subarray}}\mu_{\max}(a,z)=:\mu_A(x,z),\]
which proves \eqref{Equ: muA(y,z)geq muA(x,z)}.
In the rest of the proof, we assume that the pay-off function $\mu$ is convex. 

(a) Let $a\in\mathscr L$ such that $x\leqslant a<z$. If $y\leqslant a$, then by definition
\[\mu_A(y,z)\leqslant \mu_{\max}(a,z).\]

If $y\not\leqslant a$, then by Lemma \ref{Lem: majoration mu A star}, as $a \vee y \leqslant z$ one has  
\[\mu_A(x,y)\leqslant\mu_{\max}(a,z).\]
So in all cases the following inequality holds:
\[\inf\{\mu_A(x,y),\mu_A(y,z)\}\leqslant \mu_{\max}(a,z).\]
Taking the infimum with respect to $a$, we obtain
\[\inf\{\mu_A(x,y),\mu_A(y,z)\}\leqslant \mu_A(x,z).\]

(b) If $\mu_A(x,y)\geqslant\mu_{A}(y,z)$, then 
\[\inf\{\mu_A(x,y),\mu_A(y,z)\}=\mu_A(y,z),\]
thus \eqref{Equ: muA(y,z)geq muA(x,z)} and the result of (a) lead to $\mu_A(y,z)=\mu_A(x,z)$. If $\mu_A(x,y)<\mu_A(y,z)$, then one has
\[\inf\{\mu_A(x,y),\mu_A(y,z)\}=\mu_A(x,y).\]
Hence \eqref{Equ: muA(y,z)geq muA(x,z)} and (a) imply 
\[\mu_A(x,y)\leqslant\mu_A(x,z)\leqslant\mu_A(y,z).\]

(c) We first assume that $\mu_A(x,y)$ and $\mu_A(y,z)$ are comparable. By (b), we obtain that either the equality $\mu_A(y,z)=\mu_A(x,z)$ holds or one has $\mu_A(x,y)<\mu_A(y,z)$, and hence  
\[\mu_A(x,y)\leqslant\mu_A(x,z)\leqslant\mu_A(y,z).\]
Note that the last inequality is strict since $\mu_A(x,z)\neq\mu_A(y,z)$.

In the following, we suppose that there exists $a\in\mathscr L$ such that $x\leqslant a<z$ and that $\mu_A(x,z)=\mu_{\max}(a,z)$. If $y\leqslant a$, then by definition
\[\mu_A(y,z)\leqslant\mu_{\max}(a,z)=\mu_A(x,z).\]
Thus equality 
$\mu_A(y,z)=\mu_A(x,z)$ holds, thanks to \eqref{Equ: muA(y,z)geq muA(x,z)} (this is the same argument as for the proof of (a)). In particular, if $\mu_A(y,z)\neq \mu_A(x,z)$, then $y\not\leqslant a$ and thus, by Lemma \ref{Lem: majoration mu A star}, one has 
\[\mu_{A}(x,y)\leqslant\mu_{\max}(a,z)=\mu_A(x,z).\] 
Combining this with \eqref{Equ: muA(y,z)geq muA(x,z)} and the condition $\mu_A(y,z)\neq \mu_A(x,z)$, we obtain
\[\mu_A(x,y)\leqslant\mu_A(x,z)<\mu_A(y,z).\]

\end{proof}

\begin{rema}
Assume that $(S,\leqslant)$ is a totally ordered set and that the pay-off function $\mu$ is convex. Let $x$ be an element of $\mathscr L\setminus\{\perp,\top\}$ such that $\mu_A(x)>\mu_A(\top)$. By Proposition \ref{Pro: suite exacte et mu A} (c), we obtain $\mu_A(x,\top)=\mu_A(\top)$. This equality shows that, in the case where Alice goes first, it is not less favorable for her to consider elements of $\mathscr L_{[x,\top]}\setminus\{\top\}$ than to consider all elements of $\mathscr L\setminus\{\top\}$.
\end{rema}

\begin{prop}\phantomsection\label{Pro: mu A star inequality}
Assume that the function $\mu$ is convex. Let $u$, $x$ and $y$ be elements of $\mathscr L$ such that $u< x$ and $u<y$. 
\begin{enumerate}[label=\rm(\alph*)]
\item The following inequality holds:
\[\mu_A(u,x\vee y)\geqslant\inf\{\mu_A(u,x),\mu_A(u,y)\}.\]

\item If $\mu_A(u,x)$ and $\mu_A(u,y)$ are comparable, or if  the infimum
\[\mu_A(u,x\vee y)=\inf_{\begin{subarray}{c}a\in\mathscr L\\
u\leqslant a<x\vee y
\end{subarray}}\mu_{\max}(a,x\vee y) \]
is attained, then, either $\mu_A(u,x\vee y)\geqslant\mu_A(u,x)$ or $\mu_A(u,x\vee y)\geqslant\mu_A(u,y)$. 
\end{enumerate}
\end{prop}
\begin{proof}
Let $w$ be an element of $\mathscr L$ such that $u\leqslant w<x\vee y$. Since $w<x\vee y$, either $x\not\leqslant w$, or $y\not\leqslant w$. In the case where $x\not\leqslant w$, since $x\vee w\leqslant x\vee y$ and $u\leqslant x\wedge w$, by Lemma \ref{Lem: majoration mu A star} we obtain 
\[\mu_A(u,x)\leqslant\mu_{\max}(w,x\vee y).\]
Similarly, in the case where $y\not\leqslant w$, one has 
\[\mu_A(u,y)\leqslant \mu_{\max}(w,x\vee y).\]

(a) The above argument shows that, either $\mu_A(u,x)\leq \mu_{\max}(w,x\vee y)$ or $\mu_A(u,y)\leq \mu_{\max}(w,x\vee y)$. Hence the inequality \[\mu_{\max}(w,x\vee y)\geqslant\inf\{\mu_A(u,x),\mu_A(u,y)\}\]
holds, which leads to
\[\mu_A(u,x\vee y)\geqslant \inf\{\mu_A(u,x),\mu_A(u,y)\}\] by taking the infimum with respect to $w$.

(b) If $\mu_A(u,x)$ and $\mu_A(u,y)$ are comparable, then $\inf\{\mu_A(u,x),\mu_A(u,y)\}$ is either equal to $\mu_A(u,x)$ or equal to $\mu_A(u,y)$. Hence the statement of (a) implies that, either $\mu_A(u,x\vee y)\geqslant\mu_A(u,x)$ or $\mu_A(u,x\vee y)\geqslant\mu_A(u,y)$. 

In the following, we assume that there exists $a\in\mathscr L$ such that $u\leqslant a<x\vee y$ and that $\mu_A(u,x\vee y)=\mu_{\max}(a,x\vee y)$. By the argument in the beginning of the proof, we obtain that, either $\mu_A(u,x)\leqslant\mu_{\max}(a,x\vee y)=\mu_A(u,x\vee y)$, or $\mu_A(u,y)\leqslant\mu_{\max}(a,x\vee y)=\mu_A(u,x\vee y)$. The statement is thus proved.

\end{proof}

\section{Semi-stability and Harder-Narasimhan filtration}

In this section, we introduce the notion of semi-stability and define the Harder-Narasimhan filtration of Harder-Narasimhan games. As in the previous section, we assume that $(\mathscr L,\leqslant)$ is a bounded lattice, and we fix a complete lattice $(S,\leqslant)$ together with a map $\mu:P_{<}(\mathscr L)\rightarrow S$ which  is assumed to be convex (see Definition \ref{Def: convex pay-off}).

\subsection{Semi-stability}

\begin{defi}
We say that the ordered set $(\mathscr L,\leqslant)$ satisfies the \emph{ascending chain condition} if every non-empty subset of $\mathscr L$ admits a maximal element with respect to $\leqslant$, or equivalently, any strictly ascending chain of elements of $\mathscr L$ is finite. 

We say that $(\mathscr L,\leqslant)$ satisfies the \emph{$\mu_A$-descending chain condition} if, for any $a\in\mathscr L$, there does not exist any infinite descending chain
\[x_0> x_1>\ldots> x_n> x_{n+1}>\ldots\]
of elements of $\mathscr L$ that are  $>a$, such that 
\[\mu_A(a,x_0)<\mu_A(a,x_1)<\ldots<\mu_A(a,x_n)<\mu_A(a,x_{n+1})<\ldots.\] Clearly, if $(\mathscr L,\leqslant)$ satisfies the $\mu_A$-descending chain condition, then, for any $x\in\mathscr L\setminus\{\top\}$, $(\mathscr L_{[x,\top]},\leqslant)$ satisfies the $\mu_{[x,\top],A}$-descending chain condition.
\end{defi}

The following proposition helps to describe a typical situation where the $\mu_A$-descending chain condition is satisfied (compare with Remark \ref{Rem: strong descending chain condition}). We denote  the greatest element of $(S,\leqslant)$ as $+\infty$.

\begin{prop}\label{Pro: mu A descending chain} 
Let $(x,z)$ be a pair in $P_{<}(\mathscr L)$. If $\mu_A(x,z)=+\infty$, or equivalently, for any $w\in\mathscr L$ such that $x\leqslant w<z$ one has $\mu_{\max}(w,z)=+\infty$, then, for any $a\in\mathscr L$ such that $a<x$, one has $\mu_A(a,x)\leqslant\mu_A(a,z)$. 
\end{prop}
\begin{proof} Since $\mu_A(x,z)=+\infty$, for any $a\in\mathscr L$ such that $a<x$, one has 
\[\inf\{\mu_A(a,x),\mu_A(x,z)\}=\mu_A(a,x).\]
By Proposition \ref{Pro: suite exacte et mu A} (a), we obtain $\mu_A(a,z)\geqslant\mu_A(a,x)$.
\end{proof}

\begin{coro}\label{Cor: strong muA descending}
If, for any descending chain 
\[x_0>x_1>\ldots>x_n>x_{n+1}>\ldots\] 
of elements of $\mathscr L$, there exists $N\in\mathbb N$ such that 
$\mu_A(x_{N+1},x_N)=+\infty$,
then $(\mathscr L,\leqslant)$ satisfies the $\mu_A$-descending chain condition.
\end{coro}
\begin{proof}
This is a direct consequence of  Proposition \ref{Pro: mu A descending chain}.
\end{proof}

\begin{prop}\phantomsection\label{Pro: destabilizing element} Assume that  $(\mathscr L,\leqslant)$ satisfies the ascending chain condition and the $\mu_A$-descending chain condition. There exists an element $x\in\mathscr L\setminus\{\perp\}$ that satisfies the following conditions: 
\begin{enumerate}[label=\rm(S\arabic*)]
\item For any $y\in\mathscr L\setminus\{\perp\}$, $\mu_A(y)\not>\mu_A(x)$.
\item For any $y\in\mathscr L\setminus\{\perp\}$, if $\mu_A(y)=\mu_A(x)$, then $y\leqslant x$.
\end{enumerate}
\end{prop}

\begin{proof}
We first construct  a strictly decreasing sequence 
\[\top=x_0>x_1>\ldots>x_n\operatorname{=}\perp, \quad n\in\mathbb N_{\geqslant 1}\]
in a recursive way, 
which satisfies the following conditions (where the second condition is ensured by the ascending chain condition):
\begin{enumerate}[label=\rm(\arabic*)]
\item For any $i\in\mathbb N$ such that $1\leqslant i<n$ one has $\mu_{A}(x_{i})>\mu_A(x_{i-1})$.  

\item\label{Item: mumin yi smaller mu min xi} For any $i\in\mathbb N$ such that $1\leqslant i<n $ and any $y_i\in\mathscr L$ such that $x_{i}<y_i\leqslant x_{i-1}$, one has $\mu_{A}(y_i)\not\geqslant\mu_{A}(x_{i})$.

\item For any $y\in\mathscr L$ such that $\operatorname{\perp}<y\leqslant x_{n-1}$, one has $\mu_A(y)\not>\mu_A(x_{n-1})$.
\end{enumerate} 
We proceed as follows: first, we consider the subset \[\mathscr L_1 := \{y \in \mathscr L \setminus \{ \perp \} \mid \mu_A(y)>\mu_A(\top)\},\] which is either empty or admits a maximal element (as $(\mathscr L,\leqslant)$ satisfies the ascending chain condition). 
\begin{itemize}
\item If $\mathscr L_1 $ is empty, then one has $x_1 = \operatorname{\perp}$ in the aforementioned sequence (which is therefore of trivial shape $\top=x_0>\operatorname{\perp}=x_1$), so that only the condition (3) makes sense here and is indeed satisfied (by construction);
\item If $\mathscr L_1 $ is not empty, let $x_1$ be one of its maximal elements. One has $\top = x_0 >x_1>\operatorname{\perp}$, and by the definition of $x_1$, the inequality $\mu_{A}(x_{1})>\mu_A(x_{0})$ is  satisfied. The condition \ref{Item: mumin yi smaller mu min xi} is also satisfied for $i=1$,  by the maximality of $x_1$. Then one considers \[\mathscr L_2 := \{y \in \mathscr L \setminus \{ \perp \} \mid \mu_A(y)>\mu_A(x_1)\}\] and iterates the above reasoning. Note that the process terminates as $(\mathscr L,\leqslant)$ satisfies the $\mu_A$-descending chain condition.
\end{itemize} 

We can now show, by induction on the length of the above sequence, the existence of an element $x \in \mathscr L$ that satisfies conditions (S1) and (S2).
In the case where $n=1$, by the point (3) above, one has  
\[\forall\,y\in\mathscr L\setminus\{\perp\},\quad \mu_{A}(y)\not>\mu_{A}(\top).\]
Hence $x_0=\top$ satisfies the condition (S1), while the condition (S2) is trivially satisfied.

Suppose that $n>1$. Let $y$ be an element of $\mathscr L\setminus\{\perp\}$. We  first show by contradiction that, for any $i\in\{1,\ldots,n-1\}$, if $y<x_{i-1}$ and if $\mu_{A}(y)\geqslant\mu_{A}(x_{i})$, then $y\leqslant x_{i}$. In fact, if $y\not\leqslant x_{i}$, then one has 
\[x_{i}<y\vee x_{i}\leqslant x_{i-1}.\]
By Proposition \ref{Pro: mu A star inequality}, we obtain 
\[\mu_A(y\vee x_{i})\geqslant\inf\{\mu_A(y),\mu_{A}(x_{i})\}=\mu_A(x_{i}),\]
which contradicts the point (2) above. Now, if $y$ is an element of $\mathscr L\setminus\{\perp\} $ such that $\mu_A(y)\geqslant\mu_{A}(x_{n-1})$, then by the above statement, we can show by induction on $i$ that $y\leqslant x_{i}$ for any $i\in\{0,\ldots,n-1\}$. The case where $i=0$ is trivial. Assume that $y\leqslant x_{i-1}$ with $i\in\{1,\ldots,n-1\}$. Since
\[\mu_A(y)\geqslant\mu_A(x_{n-1})>\mu_A(x_{i-1}),\]
one has $y<x_{i-1}$. Hence the condition \[\mu_A(y)\geqslant\mu_A(x_{n-1})\geqslant\mu_A(x_i)\] leads to $y\leqslant x_i$.

The above induction argument shows that $y\leqslant x_{n-1}$. By the condition (3) one has $\mu_A(y)\not>\mu_A(x_{n-1})$. Therefore, $x_{n-1}$ satisfies the conditions (S1) and (S2). 
\end{proof}

\begin{rema}\phantomsection\label{Rem: uniqueness of destabilizing element}
Assume that $(S,\leqslant)$ is a totally ordered set. Then there is at most one element element $x\in\mathscr L\setminus\{\perp\}$ which satisfies the conditions (S1) and (S2) in the above proposition. In fact, if $x$ and $x'$ are two elements of $\mathscr L$ that satisfy the conditions (S1) and (S2), then by the condition (S1) we obtains $\mu_A(x)\leqslant \mu_A(x')$ and $\mu_A(x')\leqslant\mu_A(x)$. This implies the equality $\mu_A(x)=\mu_A(x')$. Thus by the condition (S2) we obtain $x=x'$.
\end{rema}

\begin{defi}\label{Def: semistability game}
We denote by $\operatorname{St}(\mu)$ the set of elements $x\in\mathscr L\setminus\{\perp\}$ that satisfy the conditions (S1) and (S2) of Proposition \ref{Pro: destabilizing element}. If $\top\in\operatorname{St}(\mu)$, or equivalently, if for any $x\in\mathscr L$, $\mu_A(x)\not>\mu_A(\top)$,  we say that the Harder-Narasimhan game with pay-off function $\mu$ is \emph{semi-stable}, or simply that the pay-off function $\mu$ is \emph{semi-stable}. By Remark \ref{Rem: uniqueness of destabilizing element}, in the case where $(S,\leqslant)$ is a totally ordered set, this is equivalent to requiring that $\operatorname{St}(\mu)=\{\top\}$.
 
If the Harder-Narasimhan game with pay-off function $\mu$ is semi-stable and, moreover, for any $x\in\mathscr L$, $\mu_A(x)\neq \mu_A(\top)$, then we say that the Harder-Narasimhan game with pay-off function $\mu$ is \emph{stable}, or simply that the pay-off function $\mu$ is \emph{stable}. 
\end{defi}


\begin{prop}\phantomsection\label{Pro: troncation}
Let $x$ be an element of $\operatorname{St}(\mu)\setminus\{\top\}$.  Then the following assertions hold.
\begin{enumerate}[label=\rm(\arabic*)]
\item The function $\mu_{[\perp,x]}$ is semi-stable.
\item For any $y\in\mathscr L$ such that $y>x$, one has $\mu_A(x)\not\leqslant\mu_A(x,y)$. 
\end{enumerate}
\end{prop}
\begin{proof}

(1) Let $\nu=\mu_{[\perp,x]}$. For any $z\in\mathscr L\setminus\{\perp\}$ such that $z\leqslant x$, one has 
$\nu_A(z)=\mu_A(z)$.
Hence, for any $z\in\mathscr L_{[\perp,x]}$, one has $\nu_A(z)\not>\nu_A(x)$, and $\nu_A(z)=\nu_A(x)$ implies that $z\leqslant x$. Therefore $x\in\operatorname{St}(\nu)$.

(2) If $\mu_A(x)\leqslant\mu_A(x,y)$, then by Proposition \ref{Pro: suite exacte et mu A}, one obtains $\mu_A(y)\geqslant\mu_A(x)$, which leads to $y\leqslant x$ since $x\in\operatorname{St}(\mu)$. This contradicts the condition $y>x$. 
\end{proof}

\begin{prop}\label{Pro: property of St}
If $(S,\leqslant)$ is a totally ordered set, or if the infimum 
\[\mu_A(z)=\inf_{\begin{subarray}{c}
a\in\mathscr L\\
\perp\leqslant a<z
\end{subarray}}\mu_{\max}(a,z)\]
is attained for any $z\in\mathscr L\setminus\{\perp\}$, then the following statements hold.
\begin{enumerate}[label=\rm(\arabic*)]
\item The subset $\operatorname{St}(\mu)$ of $(\mathscr L,\leqslant)$ is totally ordered, and hence admits a greatest element if $(\mathscr L,\leqslant)$ satisfies the ascending chain condition.

\item If $x$ is an element of $\operatorname{St}(\mu)$ and  $y$ is an element of $\mathscr L$ such that $x<y$, then
 $\mu_A(y)=\mu_A(x,y)$.  
\end{enumerate}
\end{prop}

\begin{proof}
(1) Let $x$ and $x'$ be two elements of $\operatorname{St}(\mu)$. By Proposition \ref{Pro: mu A star inequality} (b), we obtain that, either $\mu_A(x\vee x')\geqslant\mu_A(x)$, or $\mu_A(x\vee x')\geqslant\mu_A(x')$. By the condition (S1) we deduce that, either $\mu_A(x\vee x')=\mu_A(x)$, or $\mu_A(x\vee x')=\mu_A(x')$. Hence the condition (S2) leads to, either $x\vee x'\leqslant x$, or $x\vee x'\leqslant x'$, that is, $x'\leqslant x$ or $x\leqslant x'$.

(2) By Proposition \ref{Pro: suite exacte et mu A} (c), either $\mu_A(x,y)=\mu_A(y)$, or $\mu_A(x)\leqslant\mu_A(y)<\mu_A(x,y)$. However, the latter condition cannot hold since otherwise, $\mu_A(y)=\mu_A(x)$ by the condition (S1), and further, $y\leqslant x$ by the condition (S2). Therefore $\mu_A(x,y)=\mu_A(y)$.
\end{proof}

\subsection{Illustration by weighted directed graphs}\label{Sec: wdg} 
One can illustrate the pay-off function of a Harder-Narasimhan game as an edge-weighted directed graph. Let $(\mathscr L,\leqslant)$ be a bounded poset, $(S,\leqslant)$ a complete lattice, and $\mu:P_{<}(\mathscr L)\rightarrow S$ be a map. We can construct a weighted directed graph $G(\mathscr L,\mu)$ as follows. The vertices of $G(\mathscr L,\mu)$ are elements of $\mathscr L$. For any $(x,y)\in P_{<}(\mathscr L)$, one connects $x$ to $y$ by a directed edge with weight $\mu(x,y)$. For $(a,b)\in S^2$, we denote by $a\vee b$ the supremum of $\{a,b\}$ and by $a\wedge b$ the infimum of $\{a,b\}$.

The simplest non-trivial Harder-Narasimhan game can be illustrated by the following weighted directed graph, where $a$, $b$ and $c$ are elements of $S$. 
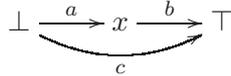
\begin{equation*}\xymatrix{\relax\perp\ar@/_1pc/[rr]_-{c}\ar[r]^-a&x\ar[r]^-b&\top}\end{equation*}
Note that
\[\mu_A^*:= \inf_{x\in\mathscr L\setminus\{\top\}}\sup_{\begin{subarray}{c}y\in\mathscr L\\
x<y\end{subarray}}\mu(x,y) =(a\vee c)\wedge b = \mu_A(\top),\]
\[ \mu_A(x)=a,\quad \mu_A(x,\top)=b.\] 
Therefore, the Harder-Narasimhan game is semi-stable if and only if 
\[\mu_A(x) := \inf_{x\in\mathscr L\setminus\{\top\}}\sup_{\begin{subarray}{c}y\in\mathscr L\\
x<y\end{subarray}}\mu(x,y) = a\not> (a\vee c)\wedge b =\mu_A(\top).\]
Note that, if $a\leqslant b$, then 
\[a\leqslant(a\vee c)\wedge b\]
and hence the Harder-Narasimhan game is semi-stable. If $a>b$, then \[(a\vee c)\wedge b=b\] and hence the Harder-Narasimhan game is not semi-stable and $\operatorname{St}(\mu)=\{x\}$.

Assume that $a$ and $b$ are not comparable, but $a\vee c\geqslant b$ (this happens for example when $c\geqslant b$). In this case one has $(a\vee c)\wedge b=b$. Therefore,  $\mu_A(\top)$ and $\mu_A(x)$ are not comparable, and hence $\operatorname{St}(\mu)=\{x,\top\}$. In this case the Harder-Narasimhan game is still semi-stable.

\subsection{Harder-Narasimhan filtration}

In this subsection, we consider the Harder-Narasimhan filtration of a Harder-Narasimhan game, which is a canonical filtration (in a certain sense) that measures the potential default of semi-stability of a Harder-Narasimhan game. We fix a bounded lattice $(\mathscr L,\leqslant)$ and a pay-off function $\mu$ on $P_{<}(\mathscr L)$ which takes value in a complete lattice $(S,\leqslant)$. We assume the following:
\begin{enumerate}[label=(\rm\alph*)]
\item The pay-off function $\mu$ is convex.
\item The bounded lattice $(\mathscr L,\leqslant)$ satisfies the ascending chain condition and the $\mu_A$-descending chain condition.
\item Either $(S,\leqslant)$ is totally ordered, or the infimum
\[\mu_A(x,y)=\inf_{\begin{subarray}{c}
a\in\mathscr L\\
x\leqslant a<y
\end{subarray}}\mu_{\max}(a,y)\]
is attained for any $(x,y)\in P_{<}(\mathscr L)$. 
\end{enumerate}

\begin{defi}\label{Def: HN filtration}
We construct an increasing sequence 
\[\operatorname{\perp}=a_0<a_1<\ldots<a_n=\top\]
in a recursive way, so that $a_i$ is the greatest element of the set $\operatorname{St}(\mu_{[a_{i-1},\top]})$ for any $i\in\{1,\ldots,n\}$ (the existence of $a_i$ is ensured by Proposition \ref{Pro: property of St}). Note that, by the ascending chain condition, this sequence has a finite length. We call it the \emph{Harder-Narasimhan filtration} of $\mu$. By Proposition \ref{Pro: troncation}, we obtain that each $\mu_{[a_{i-1},a_i]}$ is semi-stable, and \[\mu_A(a_0,a_1)\not\leqslant \mu_A(a_1,a_2)\not\leqslant\ldots\not\leqslant\mu_A(a_{n-1},a_n).\]
Moreover, by Proposition \ref{Pro: property of St}, for any $i\in\{0,\ldots,n-1\}$ and any $y\in\mathscr L$ such that $x_{i}<y$, one has 
\[\mu_A(\perp,y)=\mu_A(x_1,y)=\cdots=\mu_A(x_{i},y).\] 
\end{defi}

\begin{theo}\label{Thm: uniqueness of HN filtration}
Suppose that $(S,\leqslant)$ is totally ordered. Let 
\[\operatorname{\perp}=b_0<b_1<\ldots<b_m=\top\]
be a strictly increasing sequence in $\mathscr L$ such that, for any $j\in\{1,\ldots,m\}$, the restriction $\mu_{[b_{j-1},b_j]}$ is semi-stable, and  that 
\[\mu_A(b_0,b_1)>\ldots>\mu_A(b_{m-1},b_m).\]
Then this sequence identifies with the Harder-Narasimhan filtration of $\mu$.
\end{theo}
\begin{proof}
We denote by $\operatorname{\perp}=a_0<a_1<\ldots<a_n=\top$ the Harder-Narasimhan filtration of $\mu$, and we proceed by induction on $n$. Let $i\in\{1,\ldots,m\}$ be the smallest index such that $a_1\leqslant b_i$.
Since $a_1\not\leqslant b_{i-1}$, by Lemma \ref{Lem: majoration mu A star} one has 
\[\mu_A(a_1)\leqslant\mu_{A}(b_{i-1},a_1\vee b_{i-1}).\]
By the semi-stability of $\mu_{[b_{i-1},b_i]}$ and the hypothesis that $(S,\leqslant)$ is totally ordered, we deduce that 
\[\mu_A(a_1)\leqslant\mu_A(b_{i-1},b_i).\]
If $i>1$, then one has $\mu_A(a_1)< \mu_A(b_1)$, which leads to a contradiction (as by construction $a_1$ is the greatest element of $\operatorname{St}(\mu)$). Therefore we obtain $a_1\leqslant b_1$ and hence $\mu_A(a_1)\leqslant\mu_A(b_1)$ since $\mu_{[\perp,b_1]}$ is semi-stable. As $\operatorname{St}(\mu)=\{a_1\}$, we obtain $b_1\leqslant a_1$ and hence $a_1=b_1$. If $n\geqslant 2$, we can iterate the above argument to $\mu_{[a_i,\top]}$ successively for $i\in\{2,\ldots,n\}$ to show that $n=m$ and $a_i=b_i$ for any $i\in\{1,\ldots,n\}$. The theorem is thus proved.
\end{proof}

\subsection{Coprimary filtration}
\label{sec : coprimary filtration}
Let $R$ be a N\oe therian ring and $M$ be a non-zero $R$-module of finite type. Let $\mathscr L$ be the set of all sub-$R$-modules of $M$. The set $\mathscr L$, equipped with the order of inclusion $\subseteq$, forms a bounded lattice. Its least element is the zero $R$-module $\boldsymbol{0}$, and its greatest element is $M$. If $N_1$ and $N_2$ are two elements of $\mathscr L$, the least upper bound of $\{N_1,N_2\}$ is their sum, the greatest lower bound of $\{N_1,N_2\}$ is their intersection. 

We equip $\operatorname{Spec}(R)$ with an arbitrary total order $\leqslant$ which extends the partial order of inclusion $\subseteq$. Let $S_0$ be the set of finite subsets of $\operatorname{Spec}(R)$. We equip $S_0$ with a total order $\leqslant$ which extends the order of inclusion and such that, for any $(\mathfrak p,\mathfrak q)\in\operatorname{Spec}(R)\times\operatorname{Spec}(R)$, the relation $\{\mathfrak p\}\leqslant\{\mathfrak q\}$ holds if and only if $\mathfrak p\leqslant\mathfrak q$ in $(\operatorname{Spec}(R),\leqslant)$. Note that the lexicographic order is an example of such a total order on $S_0$. Denote by $(S,\leqslant)$ be the Dedekind-MacNeille completion of $(S_0,\leqslant)$. 


For any pair $(N',N)$ of sub-$R$-modules of $M$ such that $N'\subsetneq N$, we define \[\mu(N',N):=\Ass(N/N')\] as the set of associated prime ideals of the quotient module $N/N'$. Recall that a prime ideal $\mathfrak p$ of $R$ is said to be an associated prime ideal of $N/N'$ if there exists $s\in N$ such that 
\[\mathfrak p=\{a\in R\,:\,as\in N'\}.\]

Clearly, if $N''$ is a sub-$R$-module of $N$ that strictly contains  $N'$, then any associated prime ideal of $N''/N'$ is also an associated prime ideal of $N/N'$. Thus one has 
\[\mu_{\max}(N',N)=\mu(N',N).\]

\begin{prop}
The above map $\mu$ is convex as a pay-off function of a Harder-Narasimhan game.
\end{prop}
\begin{proof}
Let $N_1$ and $N_2$ be two sub-$R$-modules of $M$ such that $N_1\not\subseteq N_2$. One has
\[N_1/(N_1\cap N_2)\cong (N_1+N_2)/N_2\]
as $R$-modules. Therefore we obtain
\[\mu(N_1\cap N_2,N_1)=\mu(N_2,N_1+N_2).\]
\end{proof}

\begin{prop}\label{Pro: muA N'N}
Let $N'$ and $N$ be two sub-$R$-modules of $M$ such that $N'\subsetneq N$, and $\mathfrak p$ be the least element of $\mu(N',N)$ in the totally ordered set $(\operatorname{Spec}(R),\leqslant)$. Then the equality $\mu_A(N',N)=\{\mathfrak p\}$ holds. 
\end{prop}
\begin{proof}
Recall that the support of an $R$-module $N$ is the set \[\Supp(N) := \{\mathfrak{q} \in \Spec(R) \mid N_{\mathfrak{q}} \neq 0\},\] where $N_{\mathfrak{q}}$ is the localisation at $\mathfrak{q}$. Let $N''$ be a sub-$R$-module of $N$ such that $N'\subseteq N''\subsetneq N$. Note that the support of $N/N''$ is contained in the support of $N/N'$. If $\mathfrak q$ is an associated prime ideal of $N/N''$, then it belongs to the support of $N/N'$ since \[\Ass(N/N'') \subseteq \Supp(N/N'')\subseteq \Supp(N/N').\] This follows from the fact that $R$ is N\oe therian (see \cite[chap. IV, \S1, no.3 Corollary 1]{MR782296}). In particular, there exists a minimal (under the relation of inclusion) prime ideal of $\operatorname{Supp}(N/N')$ which is contained in $\mathfrak q$. Hence, in the totally ordered set $(\Spec R,\leqslant)$, $\mathfrak q$ is greater or equal to $\mathfrak p$. Moreover, by \cite[chap. IV, \S1, no.2 Proposition 6]{MR782296}, if we take $N''$ as the kernel of the canonical $R$-module homomorphism
\[N\longrightarrow N/N'\longrightarrow (N/N')_{\mathfrak p},\]
then $\mathfrak p$ is the only associated prime ideal of $N/N''$. Therefore, we obtain that 
\[\mu_{A}(N',N)=\mu_{\max}(N'',N)=\mu(N'',N)=\{\mathfrak p\}.\]
\end{proof}

\begin{prop}
The bounded lattice $(\mathscr L,\subseteq)$ satisfies the ascending chain condition and the $\mu_A$-descending chain condition.
\end{prop}
\begin{proof}
The first statement comes from the hypotheses that $R$ is a N\oe therian ring and $M$ is an $R$-module of finite type (and hence a N\oe therian $R$-module). In the following, we prove the second statement. Let $N$ be a sub-$R$-module of $M$ such that $N\subsetneq M$. Assume that  
\[M_0\supsetneq M_1\supsetneq\ldots\supsetneq M_n\supsetneq M_{n+1}\supsetneq\ldots\]
is a decreasing sequence of sub-$R$-modules
of $M$ such that $N\subsetneq M_n$ for any $n\in\mathbb N$. Since each $\mu_A(N,M_i)$ is a one point subset of $\operatorname{Ass}(M_i/N)\subseteq \operatorname{Ass}(M_0/N)$, it only has finitely many possible values. Therefore the inequalities  
\[\mu_A(N,M_0)<\mu_A(N,M_1)<\ldots<\mu_A(N,M_n)<\mu_A(N,M_{n+1})<\ldots\]
cannot hold simultaneously.\end{proof}

\begin{rema}
We consider the semi-stability of this Harder-Narasimhan game. Let $\mathfrak p$ be the least element of the set $\mu(\boldsymbol{0},M)$ in $(\operatorname{Spec}(R),\leqslant)$. The Harder-Narasimhan game of pay-off function $\mu$ is semi-stable if and only if, for any $N\in\mathscr L\setminus\{\boldsymbol{0}\}$, one has $\mu_A(\boldsymbol{0},N)=\{\mathfrak p\}$, or equivalently, $\mathfrak p$ is an associated prime ideal of $N$. Therefore, we obtain that the semi-stability of the Harder-Narasimhan game is equivalent to the condition that $M$ has only one associated prime ideal (namely $M$ is coprimary).
\end{rema}

\begin{theo}\label{Thm: coprimary sequence}
Let $M$ be a non-zero $R$-module of finite type. There exists a unique sequence 
\[0=M_0\subsetneq M_1\subsetneq M_2\subsetneq\ldots\subsetneq M_n=M\]
of sub-$R$-modules of $M$ and a sequence 
\[\mathfrak p_1>\ldots>\mathfrak p_n\]
of prime ideals in $\operatorname{Spec}(R)$ such that each subquotient $M_i/M_{i-1}$ is $\mathfrak p_i$-coprimary.
\end{theo}
\begin{proof}
This is a direct consequence of Theorem \ref{Thm: uniqueness of HN filtration}.
\end{proof}

\begin{rema}
In the above theorem, the set $\{\mathfrak p_1,\ldots,\mathfrak p_n\}$ identifies with the set $\operatorname{Ass}(M)$ of associated prime ideals of $M$. In fact, by \cite[chap. IV, \S1, no.4 Théorème 2]{MR782296}, one has 
\[\operatorname{Ass}(M)=\mu(\boldsymbol{0},M)\subseteq\{\mathfrak p_1,\ldots,\mathfrak p_n\}.\]
To show the reverse inclusions, remember that by Definition \ref{Def: HN filtration}, for any $i \in\{1, \ldots, n\}$ one has $\mu_A(M_{i})= \mu_A(M_{i-1},M_{i}) = \{\mathfrak{p}_{i}\}$, where $\mathfrak{p}_i$ is the least element of $\operatorname{Ass}(M_{i-1}, M_i)$ according to Proposition \ref{Pro: muA N'N}. 

Note that the existence of filtrations of $M$ with $\mathfrak{p}_i$-coprimary quotients is well-known in the literature (see for instance \cite[chap. IV, \S1, no.4 Théorème 1 and Théorème 2]{MR782296}). However, as noticed after the corollary that follows \cite[chap. IV, \S1, no.4 Théorème 2]{MR782296} the $\mathfrak{p}_i$'s are not necessarily uniquely determined by $M$. Nevertheless, the choice of a total order on $\Spec(R)$ allows one to get a more canonical filtration: while for the aforementioned well-known filtrations one only had the inclusion $\operatorname{Ass}(M) \subseteq \{\mathfrak p_1,\ldots,\mathfrak p_n\}$, the choice of a total order allows to obtain the Harder-Narasimhan filtration of Theorem \ref{Thm: coprimary sequence} by ordering $\operatorname{Ass}(M) \subseteq \{\mathfrak p_1,\ldots,\mathfrak p_n\}$ and choosing these ordered elements as the $\mathfrak{p}_i$'s, in decreasing order. This is \cite[chap. IV, Exercises, \S2, Exercise 7]{MR782296}. This also clarify the impact the choice of a total order on $\Spec(R)$ has: it will permute the $\mathfrak{p}_i$'s, but the number of factors in the filtration will remain the same.
\end{rema}

\subsection{Reductive group schemes over a curve}
\label{Sec: red gps over curve}

Let $C$ be a regular projective curve defined over a field $k$, and let $K$ be the function field of $C$. For any smooth affine algebraic group scheme with connected fibers $H$ over $C$, one defines the degree of $H$ to be that of its Lie algebra seen as a vector bundle over $C$. 

Let $G$ be a reductive group scheme over $C$, then $\deg(G) = 0$ (see \cite[Note 4.2]{BEH}), and the degrees of parabolic subgroups $P\subseteq G$ are bounded from above by a non-negative integer (this is a consequence of Riemann-Roch Theorem, see \cite[Lemma 4.3]{BEH}). A reductive group is semi-stable if $\deg(P)\leq 0$ for every parabolic subgroup $P\subseteq G$ (see \cite[Definition 4.4]{BEH}). The largest integer $d$ for which there exists a parabolic subgroup $P\subseteq G$ of degree $d$ is the degree of instability of $G$, denoted $d_i$. 

Fix a Borel subgroup $\tilde{B}\subseteq G$ as well as an adapted pinning for $B$. We consider the set $\mathscr L$ of parabolic subgroups of $G$ and endow $\mathscr L$ with the order induced by the type of the parabolic (see for instance \cite[XXVI 3.2 and 3.7]{SGA33}). For this order $\mathscr L$ has a greatest element, namely $G$, and a least element, namely $\tilde{B}$. Note that here, any pair of elements indeed has a maximal and a minimal element to any parabolic couple $(P,Q)$ corresponds a pair $(P', Q')$ of parabolic subgroups of same type of $P$, respectively $Q$, that contain $\tilde{B}$ and the minimal, respectively maximal element of $(P,Q)$ is that $(P', Q')$. For any couple of parabolic subgroups $(P, Q) \in \mathscr L$ we let $\mu(P, Q)= \deg(Q)$. Hence
\[\mu_A(P,Q) = \inf_{P\subseteq P' \subsetneq Q} \sup_{P'\subsetneq Q' \subseteq Q} \deg(Q'),\]
so in particular
\[\mu_A(G_0,G) = \inf_{G_0\subseteq P' \subsetneq G} \sup_{P'\subsetneq Q' \subseteq G} \deg(Q') = 0,\]
because \[\sup_{P'\subsetneq Q' \subseteq G} \deg(Q') \geqslant 0\] since $\deg(G) =0$. 
Therefore $G$ is semi-stable if and only if $\deg(P) \leqslant 0$ for every parabolic subgroup $P\subseteq G$. Indeed if $G$ is semi-stable then \[\mu_A(P) := \mu_A(\tilde{B}, P)\leq \mu_A(\tilde{B}, G) = 0.\] Hence 
\[\inf_{\tilde{B}\subseteq P' \subsetneq P} \sup_{P'\subsetneq Q' \subseteq P} \deg(Q')\leq 0,\]
which means, in particular, that $\deg(P)\leq 0$. The reverse implication is clear as $\deg(G)$ is always considered in $\mu_A(\tilde{B},G)$. In particular, $\mu$ is semi-stable as in Definition \ref{Def: semistability game} if and only if $G$ is semi-stable in the sense of \cite{BEH}. Moreover, according to Proposition \ref{Pro: destabilizing element}, there exists a maximal destabilising element in $\mathscr L$: the restriction of the group structure to a parabolic subgroup $P_i$ of $G$ of degree $d_i$ such that $P$ is maximal for the inclusion. Hence the machinery of Harder--Narasimhan games provides another proof of \cite[Theorem 7.3]{BEH}.

With the same notations as in the above example, let $E$ be a $G$-torsor with respect to the $\fppf$-topology. Any parabolic subgroup $P\subset G$ together with a section $\sigma : C \rightarrow E(G/P)$ defines a reduction of the structure group of $E$ to $P$, denoted as $E_P$. In particular $E_G := E$. The degree of a reduction of the structure group of $E$ to $P$ is that of the corresponding vector bundle $\sigma^{*}E\times^{\Ad}\Lie(P)$. This degree satisfies the same properties as the degree of the parabolic subgroups of $G$ introduced so far. In particular it is bounded by a positive integer, and the degree of $E$ is $0$.
 We set \[d_i(E):=\max\{\deg(\sigma^*E\times^{\Ad}\Lie(P))\mid P\subseteq G \text{ is parabolic and } \sigma: X \rightarrow E(G/P).\}\]
As above we fix a Borel subgroup $\tilde{B}$ and we consider the set of reductions of the structure group of $E$ to a parabolic subgroup of $G$ containing $\tilde{B}$, denoted as $\mathscr L_E$. We endow $\mathscr L_E$ with the order induced by that on $\mathscr L$ for parabolic subgroups of $G$. For any couple $(E_P, E_Q)$ of reductions of the structure group of $E$ to parabolic subgroups $P\subset Q$ of $G$, we let $\mu(E_P, E_Q)= \deg(E_Q)$. Hence
\[\mu_A(E_P,E_Q) = \inf_{P\subseteq P' \subsetneq Q} \sup_{P'\subsetneq Q' \subseteq Q} \deg(E_{Q'}),\]
so in particular
\[\mu_A(E_{\tilde{B}},E_G) = \inf_{\tilde{B}\subseteq P' \subsetneq G} \sup_{P'\subsetneq Q' \subset's G} \deg(E_{Q'}) = 0,\]
because $\sup_{P'\subsetneq Q' \subseteq G} \deg(E_{Q'} \geqslant 0$ as $\deg(E_G) =0$. 
Therefore $E_G$ is semi-stable if and only if $\deg(E_P) \leqslant 0$ for every reduction of group structure $(P,\beta)$ of $E$. Indeed, if $E_G$ is semi-stable then \[\mu_A(E_P) := \mu_A(E_{\tilde{B}}, E_P)\leq \mu_A(E_{\tilde{B}}, E_G) = 0.\] Hence 
\[\inf_{\tilde{B}\subseteq P' \subsetneq P} \sup_{P'\subsetneq Q' \subseteq P} \deg(E_{Q'})\leq 0,\]which means, in particular, that \[\deg(E_P):= \deg(\sigma^* E\times^{\Ad}\Lie(P))\leq 0.\] The reverse implication is clear as $\deg(E_G)$ is always considered in $\mu_A(E_{\tilde{B}},E_G)$. In particular, $\mu$ is semi-stable as defined in Definition \ref{Def: semistability game} if and only if $E_P$ is semi-stable in the sense of \cite[Proposition 4.2]{MR2483939} (note that the degrees there are positive because the author's chose a different convention). Moreover, according to \ref{Pro: destabilizing element}, there exists a maximal destabilising element in $\mathscr L$: the reduction of the group structure of $E$ to a parabolic subgroup $P_i$ of $G$ of degree $d_i$ such that $P$ is maximal for the inclusion, hence we find back \cite[Theorem 4.3]{MR2483939}.


\section{Comparison with classic Harder-Narasimhan theory}

Classic Harder-Narasimhan theory deals with the semi-stability and canonical filtration of vector bundles on a regular projective curve $C$. Let $E$ be a non-zero vector bundle on $C$, that is, a locally free $\mathcal O_C$-module of finite rank. We denote by $\mathscr L_E$ the set of all vector subbundles of $E$. This set, equipped with the relation of inclusion, forms a bounded lattice. Note that the greatest element of $\mathscr L_E$ is $E$, the least element of $\mathscr L_E$ is the zero vector bundle $\mathbf{0}$. If $F_1$ and $F_2$ are two vector subbundles of $E$, then the infimum of $\{F_1,F_2\}$ is equal to the intersection of $F_1$ and $F_2$, and the supremum of $\{F_1,F_2\}$ is equal to the sum of $F_1$ and $F_2$.

For any vector subbundle $F$ of $E$, we denote by \[\deg(F):=\deg(c_1(F)\cap [C])\] the degree of $F$, where $c_1$ is the first Chern class of $F$ and $[C]$ is the fundamental class of $C$. If $(F',F)$ is an element of $P_{\subsetneq}(\mathscr L_E)$, in the case where $\rang(F')<\rang(F)$, we let
\[\mu(F',F):=\frac{\deg(F)-\deg(F')}{\rang(F)-\rang(F')},\]
otherwise let $\mu(F',F)=+\infty$. We thus obtain a function $\mu$ on $P_{\subset}(\mathscr L_E)$ valued in the extended real number line $[-\infty,+\infty]$, which gives the pay-off function of a Harder-Narasimhan game.

Recall that the vector bundle $E$ is said to be \emph{semi-stable} if, for any non-zero vector subbundle $F \subseteq E$, one has 
\[\mu(\boldsymbol{0},F)=\frac{\deg(F)}{\rang(F)}\leqslant\frac{\deg(E)}{\rang(E)}=\mu(\boldsymbol{0},E).\]
In this section, we compare the semi-stability of Harder-Narasimhan game with this semi-stability condition by placing the problem in a general context, before addressing this specific situation in example \ref{Ex: classical HN game}. In what follows, we fix a bounded poset $(\mathscr L,\leqslant)$, a complete lattice $(S,\leqslant)$ and a map $\mu:P_{<}(\mathscr L)\rightarrow S$.

\subsection{Conditions for first-mover advantage}

Recall that
\[\mu_A^*=\inf_{x\in\mathscr L\setminus\{\top\}}\sup_{\begin{subarray}{c}y\in\mathscr L\\
x<y\end{subarray}}\mu(x,y)\]
delimits the optimal pay-off of the Harder-Narasimhan game under the hypothesis that Alice goes first. Similarly, 
\[\mu_{B}^*=\sup_{y\in\mathscr L\setminus\{\perp\}}\inf_{\begin{subarray}{c}x\in\mathscr L\\
x<y
\end{subarray}}\mu(x,y)\]
delimits the optimal pay-off of the Harder-Narasimhan game under the hypothesis that Bob goes first. Since the purpose of Alice is to minimize the pay-off and that of Bob is to maximize it, we obtain that the game favors the first mover if and only if $\mu_A^*\leqslant\mu_B^*$. In fact, if Alice goes first, then the optimal pay-off $\mu_A^*$ is not greater than that in the case where she goes second. So she has an advantage to go first since her purpose is to minimize the pay-off. For the same reason, if Bob goes first, then the optimal pay-off is not less than that in the case where he goes second, and thus he also has an advantage to go first since his aim is to maximize the pay-off. The following proposition gives a criterion of first-mover advantage. 

\begin{prop}\phantomsection\label{Pro: triangle 1}
Assume that the following conditions are satisfied:
\begin{enumerate}[label=\rm(\arabic*)]
\item\label{Item: condition 1} For any ascending chain 
\[x_0< x_1<\ldots< x_n< x_{n+1}<\ldots\]
of elements of $\mathscr L$, 
there exists $N\in\mathbb N$ such that  
$\mu(x_N,x_{N+1})\leqslant\mu(x_{N},\top)$.

\item\label{Item: condition 2} For any  $(x,y)\in\mathscr L^2$ such that $x<y<\top$, either $\mu(x,y)\leqslant\mu(x,\top)$ or $\mu(y,\top)\leqslant\mu(x,\top)$.
\end{enumerate}
Then the following equality holds
\begin{equation}\label{Equ: mu A star1}\mu_A^*:=\inf_{x\in\mathscr L\setminus\{\top\}}\sup_{\begin{subarray}{c}y\in\mathscr L\\
x<y\end{subarray}}\mu(x,y)=\inf_{x\in\mathscr L\setminus\{\top\}}\mu(x,\top).\end{equation}
In particular, one has $\mu_A^*\leqslant\mu_B^*$, namely the game favors the player who makes the first choice. 
\end{prop}
\begin{proof}
We show that, if Bob goes first and chooses the element $\top$, whatever  choice $y_A$ that Alice makes in the second step, there exists an element $x_A$ of $\mathscr L\setminus\{\top\}$ such that  
\[\forall\,x_B\in\mathscr L,\quad x_A<x_B\Longrightarrow \mu(x_A,x_B)\leqslant\mu(y_A, \top).\]
This statement signifies that Alice would be able to match or even beat the pay-off result $\mu(y_A,\top)$ if she went first.
In particular, we have
\[\mu_A^*=\inf_{x\in\mathscr L\setminus\{\top\}}\sup_{\begin{subarray}{c}y\in\mathscr L\\ x<y\end{subarray}}\mu(x,y)\leqslant \inf_{\begin{subarray}{c}
x\in\mathscr L\setminus\{\top\}\\
\end{subarray}}\mu(x,\top).\]
Since \[\forall\,x\in\mathscr L\setminus\{\top\},\quad \mu(x,\top)\leqslant\displaystyle\sup_{y\in\mathscr L, \ x<y}\mu(x,y),\] this leads to the equality \eqref{Equ: mu A star1}.

We reason by contradiction in assuming that $y_A$ is an element of $\mathscr L\setminus\{\top\}$ such that, for any $x_A\in\mathscr L\setminus\{\top\}$, there exists $x_B\in\mathscr L$ satisfying \[x_A<x_B\text{ and }\mu(x_A,x_B)\not\leqslant \mu(y_A, \top).\]
In the particular case where $x_A=y_A$, we obtain that there exists $y_B\in\mathscr L$ such that 
\begin{equation}\label{Equ: reason by contradiction1}
y_A< y_B\text{ and }\mu(y_A, y_B)\not\leqslant\mu(y_A,\top).
\end{equation}

By the condition (1), we may assume that $y_A$ is  maximal among \[\left\{y_A' \in \mathscr L \setminus \{\top\} \middle\vert \begin{array}{l}\forall\, x_A \in \mathscr L\setminus\{\top\},\; \exists\, x_B \in \mathscr L \text{ such that}\\ x_A < x_B \text{ and } \mu(x_A, x_B) \not\leqslant  \mu(y'_A, \top)\end{array}\right\}.\] Therefore, if $y_A' \in \mathscr L\setminus\{\top\}$ is such that $y_A <y_A'$ and $\mu(y_A,y_A')\not\leqslant\mu(y_A,\top)$, there exists $x_A'\in\mathscr L\setminus\{\top\}$ such that
\begin{equation}\label{Equ:for any xb1}\forall\,w\in\mathscr L,\quad x_A'<w\Longrightarrow\mu(x_A',w)\leqslant\mu(y_A', \top).
\end{equation}
By the condition (2), one then has $\mu(y_A',\top)\leqslant\mu(y_A, \top)$. However, by the hypothesis on $y_A$, for any $x_{A}' \in \mathscr L\setminus{\top}$ there exists $x_B'\in\mathscr L$ such that $x_A'<x_B'$ and $\mu(x_A',x_B')\not\leqslant\mu(y_A, \top)$. But then the inequality $\mu(y_A',\top)\leqslant\mu(y_A, \top)$ cannot hold true since otherwise \eqref{Equ:for any xb1} applied to $w = x_B'$ would give 
\[\mu(x_A',x_B')\leqslant\mu(y_A', \top)\leqslant\mu(y_A,\top),\]
which leads to a contradiction. Thus, we obtain $\mu(y_A,y_A')\leqslant\mu(y_A,\top)$ for any $y_A'\in\mathscr L$ such that $y_A< y_A'\leqslant \top$, which contradicts the condition \eqref{Equ: reason by contradiction1} for $y_A$. Therefore, the equality \eqref{Equ: mu A star1} holds. 

Finally, since
\[\mu_B^*=\sup_{y\in\mathscr L\setminus\{\perp\}}\inf_{\begin{subarray}{c}
x\in\mathscr L\\
x<y
\end{subarray}}\mu(x,y),\]
the equality \eqref{Equ: mu A star1} leads to \[\mu_A^*=\inf_{x\in\mathscr L\setminus\{\top\}}\mu(x,\top)\leqslant\mu_B^*.\]
\end{proof}

\begin{rema}\label{Rem: weak ACC}
Note that the condition (1) of Proposition \ref{Pro: triangle 1} is a weak version of the ascending chain condition, which is clearly satisfied by the bounded lattice of vector subbundles of a vector bundle on a regular projective curve.
\end{rema}

\subsection{Dual Harder-Narasimhan game}

We consider the converse order on the sets $\mathscr L$ and $S$. Note that $(\mathscr L,\geqslant)$ still forms a bounded poset, and $(S,\geqslant)$ is a complete lattice. Let $\widetilde{\mu}:P_{>}(\mathscr L)\rightarrow S$ be the map that sends a pair $(y,x)\in P_{>}(\mathscr L)$ to $\widetilde{\mu}(y,x):=\mu(x,y)$. This map defines the pay-off function of a Harder-Narasimhan game on $(\mathscr L,\geqslant)$. We call this game the \emph{dual} of the Harder-Narasimhan game on $(\mathscr L,\leqslant)$ with pay-off function $\mu$. By definition, one has
\begin{gather*}\widetilde{\mu}_B^*=\inf_{b\in\mathscr L\setminus\{\top\}}\sup_{\begin{subarray}{c}
a\in\mathscr L\\
a>b
\end{subarray}}\mu(b,a)=\mu_A^*,\\
\widetilde\mu_{A}^*=\sup_{a\in\mathscr L\setminus\{\perp\}}\inf_{\begin{subarray}{c}
b\in\mathscr L\\
a>b
\end{subarray}}\mu(b,a)=\mu_B^*
\end{gather*}
where the infima and the suprema are taken in the initial complete lattice $(S,\leqslant)$ in order to avoid confusions.

\begin{prop}\phantomsection\label{Pro: triangle 2}
Assume that the following conditions are satisfied:
\begin{enumerate}[label=\rm($\widetilde{\arabic*}$)]
\item\label{Item: tilde 1} For any descending chain 
$x_0> x_1>\ldots> x_n>x_{n+1}>\ldots$
of elements of $\mathscr L$, there exists $N\in\mathbb N$ such that
$\mu(\perp,x_N)\leqslant\mu(x_{N+1},x_N)$.

\item\label{Item: tilde 2} For any triple $(x,y)\in\mathscr L^2$ such that $\perp<x<y$, either $\mu(\perp,x)\leqslant\mu(x,y)$ or $\mu(\perp,y)\leqslant\mu(\perp,x)$.
\end{enumerate}
Then the following equality holds:
\begin{equation}\label{Equ: mu A star2}\mu_B^*:=\sup_{y\in\mathscr L\setminus\{\perp\}}\inf_{\begin{subarray}{c}
x\in\mathscr L\\
x<y
\end{subarray}}\mu(x,y)=\sup_{y\in\mathscr L\setminus\{\perp\}}\mu(\perp,y).\end{equation}
In particular, one has $\mu_A^*\leqslant\mu_B^*$, namely the dual game favors the player who makes the first choice. 
\end{prop}
\begin{proof}
This is a consequence of Proposition \ref{Pro: triangle 1} by passing to the dual Harder-Narasimhan game.
\end{proof}

\begin{rema}\phantomsection\label{Rem: strong descending chain condition}
The condition \ref{Item: tilde 1} of Proposition \ref{Pro: triangle 2}, which is the dual of the condition (1) of Proposition \ref{Pro: triangle 1}) and appears as a strong version of a descending chain condition, is notably satisfied in the following situation. Let $r:\mathscr L\rightarrow\mathbb R$ be an increasing map whose image is a well-ordered subset of $\mathbb R$. If $\mu:P_{<}(\mathscr L)\rightarrow S$ is a map such that $\mu(x,y)=+\infty$ whenever $r(x)=r(y)$, then it satisfies the condition \ref{Item: tilde 1} of Proposition \ref{Pro: triangle 2}. In fact, if 
\[x_0>x_1>\ldots>x_n>x_{n+1}>\ldots\]
is a descending chain of elements of $\mathscr L$, then there exists $N\in\mathbb N$ such that \[r(x_N)=\min\{r(x_n)
\,:\,n\in\mathbb N\}.\] Thus, $r(x_{N+1})=r(x_N)$, and hence
\[\mu(x_{N+1},x_{N})=+\infty\geqslant\mu(\perp,x_N).\] 
\end{rema}

\subsection{Slope-like pay-off function}

Remarks \ref{Rem: weak ACC} and \ref{Rem: strong descending chain condition} investigate certain conditions under which the condition (1) of Proposition \ref{Pro: triangle 1} and the condition \ref{Item: tilde 1} of Proposition \ref{Pro: triangle 2}. In this subsection, we show that, if the pay-off function $\mu$ can be expressed as the quotient of two additive functions satisfying certain mild properties (encoding a locally concave or convex behaviour), then it satisfies the condition (2) of Proposition \ref{Pro: triangle 1} and the condition \ref{Item: tilde 2} of Proposition \ref{Pro: triangle 2}.

\begin{defi}\phantomsection\label{Def: etoile gras}We say that the pay-off function $\mu$ is \emph{slope-like} if, for any $(x,y,z)\in\mathscr L^3$ such that $x<y<z$, the following four statements hold:
\begin{enumerate}[label=\rm(\arabic*)]
\item 
Either $\mu(x,y)\leqslant\mu(x,z)$, or $\mu(y,z)<\mu(x,z)$.
\item  Either $\mu(x,y)<\mu(x,z)$, or $\mu(y,z)\leqslant\mu(x,z)$.
\item Either $ \mu(x,z)<\mu(x,y)$, or $\mu(x,z)\leqslant\mu(y,z)$.
\item  Either $ \mu(x,z)\leqslant\mu(x,y)$, or $\mu(x,z)<\mu(y,z)$.
\end{enumerate}
It is not hard to check that, if $\mu$ is slope-like, then $\mu_{[a,b]}$ is also slope-like for any $(a,b)\in P_{<}(\mathscr L)$. Moreover, the condition (2) of Proposition \ref{Pro: triangle 1} and the condition \ref{Item: tilde 2} of Proposition \ref{Pro: triangle 2} can be derived respectively from conditions (1) and (4) if $\mu$ is slope-like. The following Proposition establishes a link between the slope-like condition with the seesaw property introduced by Rudakov \cite[Definition 1.1]{MR1480783}.
\end{defi}

\begin{prop}\phantomsection\label{Pro: equivalent star}
The pay-off function $\mu$ is slope-like if and only if, for any $(x,y,z)\in\mathscr L^3$ such that $x<y<z$, one \textup{(}and only one\textup{)} of the statements below is true 
\begin{enumerate}[label=\rm(\alph*)]
\item $\mu(x,y)<\mu(x,z)<\mu(y,z)$,
\item$\mu(x,y)>\mu(x,z)>\mu(y,z)$,
\item$\mu(x,y)=\mu(x,z)=\mu(y,z)$.
\end{enumerate} 
\end{prop}
\begin{proof}
``$\Longleftarrow$'': Let $(x,y,z)\in\mathscr L^3$ such that $x<y<z$. If $\mu(y,z)\not<\mu(x,z)$, then the statement (b) is not true. Thus (a) or (c) is true, which shows that $\mu(x,y)\leqslant \mu(x,z)$ holds. Hence the statement (1) in Definition \ref{Def: etoile gras} holds. The proof of other statements are similar.

``$\Longrightarrow$'': Let $(x,y,z)\in\mathscr L^3$ such that $x<y<z$. There are two possibilities:
\begin{itemize}
	\item either $\mu(x,y)<\mu(x,z)$, one has $\mu(x,y)\not\geqslant \mu(x,z)$. Hence by (4) we obtain $\mu(x,z)<\mu(y,z)$; 
	\item or $\mu(x,y)\not<\mu(x,z)$, by (2) we obtain $\mu(y,z)\leqslant \mu(x,z)$. Hence $\mu(y,z)\not>\mu(x,z)$. Furthermore, by (4) one has $\mu(x,z)\leqslant\mu(x,y)$: 
	\begin{itemize}
	\item if $\mu(x,z)<\mu(x,y)$, one has $\mu(x,z)\not\geqslant\mu(x,y)$ and hence by (1) we obtain $\mu(y,z)<\mu(x,z)$; 
	\item if $\mu(x,z)=\mu(x,y)$, then $\mu(x,z)\not<\mu(x,y)$, and by (3) we obtain $\mu(x,z)\leqslant\mu(y,z)$, which leads to the equality $\mu(x,z)=\mu(y,z)$. 
	\end{itemize}
\end{itemize}
\end{proof}

\begin{defi}
We call a \emph{totally ordered vector space over $\mathbb R$} any real vector space $V$ equipped with a total order $\leqslant$ that satisfies the following condition:
if $y$ and $z$ are two elements of $V$ such that $y\leqslant z$, then, 
\begin{enumerate}[label=\rm(\arabic*)]
\item for any $x\in V$ one has $x+y\leqslant x+z$, 
\item for any $\lambda\in\mathbb R_{\geqslant 0}$ one has $\lambda y\leqslant\lambda z$.
\end{enumerate}
\end{defi}

The following proposition shows that the pay-off function $\mu$ is slope-like if it is induced by a degree function valued in an ordered vector space.

\begin{prop}\label{Pro: degree function}
Let $(V,\leqslant)$ be a totally ordered vector space over $\mathbb R$ and $(S,\leqslant)$ be the Dedekind–MacNeille completion of $(V,\leqslant)$ \textup{(}namely the smallest complete lattice containing $(V,\leqslant)$\textup{)}. Let 
\[r:P_{<}(\mathscr L)\longrightarrow\mathbb R_{\geqslant 0}\quad\text{ and }\quad d:P_{<}(\mathscr L)\longrightarrow V \]
be two maps that satisfy the following conditions:
\begin{enumerate}[label=\rm(\roman*)]
\item\label{Item: condition i} For any $(x,y,z)\in\mathscr L^3$ with $x<y<z$, one has
\[d(x,z)=d(x,y)+d(y,z),\quad r(x,z)=r(x,y)+r(y,z).\] 
\item\label{Item: rxy=0} For any $(x,y)\in\mathscr L^2$ such that $x<y$, if $r(x,y)=0$ then $d(x,y)> 0$.
\end{enumerate}
Suppose that $\mu:P_{<}(\mathscr L)\rightarrow S$ is given by
\[\mu(x,y)=\begin{cases}
r(x,y)^{-1}d(x,y),& \text{ if }r(x,y)>0,\\
+\infty,& \text{ if } r(x,y)=0.
\end{cases}\]
Then the function $\mu$ is slope-like.
\end{prop}
\begin{proof}
Let $x$, $y$ and $z$ be three elements of $\mathscr L^3$ such that $x<y<z$. By \ref{Item: condition i}, one has 
\begin{equation}\label{Equ: degree additivity}d(x,z)=d(x,y)+d(y,z).
\end{equation}

We proceed case by case:
\begin{itemize}
 	\item In the case where $r(x,z)>0$: 
 		\begin{itemize}
 			\item if $r(x,y)> 0$ and $r(y,z)> 0$, then using that $r(x,z) = r(x,y) + r(y,z)$ we can  write \ref{Equ: degree additivity} as
\[r(x,y)\mu(x,z)+r(y,z)\mu(x,z)=r(x,y)\mu(x,y)+r(y,z)\mu(y,z),\]
which shows that one of the statements (a), (b) and (c) defined in Proposition \ref{Pro: equivalent star} is true.

		\item If $r(x,y)=0$ and $r(y,z)>0$, by \ref{Item: rxy=0} we have $d(x,y)>0$ and $\mu(x,y)=+\infty$. Hence, the inequality $\mu(x,z)<\mu(x,y)$ holds. Moreover, one has $r(x,z)=r(y,z)$ and therefore \eqref{Equ: degree additivity} leads to 
\[\mu(x,z)=\frac{d(x,z)}{r(x,z)}>\frac{d(y,z)}{r(x,z)}=\frac{d(y,z)}{r(y,z)}=\mu(y,z).\]

		\item If $r(y,z)=0$ and $r(x,y)>0$, we have $d(y,z)>0$ and $\mu(y,z)=+\infty$. Hence the inequality $\mu(x,z)<\mu(y,z)$ holds. Moreover, one has $r(x,z)=r(x,y)$ and hence \eqref{Equ: degree additivity} leads to   
\[\mu(x,z)=\frac{d(x,z)}{r(x,z)}>\frac{d(x,y)}{r(x,z)}=\frac{d(x,y)}{r(x,y)}=\mu(x,y).\]
		\end{itemize}

\item In the case where $r(x,z) =0$ then $r(x,y)=r(y,z)=0$ and one has: 
\[\mu(x,z)=\mu(x,y)=\mu(y,z)=+\infty.\]
\end{itemize}
\end{proof}

\subsection{Nash equilibrium and semi-stability}

We say that the Harder-Narasimhan game with pay-off function $\mu$ \emph{has Nash equilibrium} if $\mu_A^*=\mu_B^*$. In other words, it is equally beneficial for Alice or Bob to go first in the game. The value of $\mu_A^*=\mu_B^*$ is called the \emph{equilibrium pay-off} of the Harder-Narasimhan game.

In order to facilitate the comparison with the classic Harder-Narasimhan theory, we introduce the following notation.

\begin{enonce}{Notation}
\upshape For any $(x,y)\in P_{<}(\mathscr L)$, we denote by 
\[\mu_{\max}(x,y):=\sup_{\begin{subarray}{c}w\in\mathscr L\\
x<w\leqslant y
\end{subarray}}\mu(x,w),\quad \mu_{\min}(x,y):=\inf_{\begin{subarray}{c}w\in\mathscr L\\
x\leqslant w<y
\end{subarray}}\mu(w,y)\]
By definition, the following inequalities holds:
\[\mu_{\min}(x,y)\leqslant\mu(x,y)\leqslant\mu_{\max}(x,y).\]
For simplicity, $\mu_{\max}(\perp,y)$ is also denoted by $\mu_{\max}(y)$, and similarly $\mu_{\min}(\perp,y)$ is also denoted by $\mu_{\min}(y)$. 

\end{enonce}

\begin{rema}\phantomsection\label{Rem: mu A mu B star reinterpretation}
With the above notation, one can rewrite the optimal pay-offs $\mu_A^*$ and $\mu_B^*$ as
\begin{equation*}%
\mu_A^*=\inf_{x\in\mathscr L\setminus\{\top\}}\mu_{\max}(x,\top),\quad\mu_B^*=\sup_{y\in\mathscr L\setminus\{\perp\}}\mu_{\min}(y).
\end{equation*}
In particular, the inequality $\mu_B^*\leqslant\mu_A^*$ is equivalent to 
\[\forall\,x\in\mathscr L\setminus\{\top\},\;\forall\,y\in\mathscr L\setminus\{\perp\},\quad \mu_{\min}(y)\leqslant\mu_{\max}(x,\top).\]

If $\mu$ satisfies the conditions of Proposition \ref{Pro: triangle 1},  then one has $\mu_A^*=\mu_{\min}(\top)$, and Nash equilibrium condition can be rewritten as 
\[\forall\,y\in\mathscr L\setminus\{\perp\},\quad\mu_{\min}(y)\leqslant\mu_{\min}(\top).\]
Similarly, if $\mu$ satisfies the conditions of Proposition \ref{Pro: triangle 2}, then one has $\mu_B^*=\mu_{\max}(\top)$, and Nash equilibrium condition can be rewritten as 
\[\forall\,x\in\mathscr L\setminus\{\top\},\quad \mu_{\max}(\top)\leqslant\mu_{\max}(x,\top).\]
\end{rema}

\begin{prop}\phantomsection\label{Pro: semistable implies equilibrium1}
If 
$\mu_{\min}(\top)=\mu_{\max}(\top)$,
then the inequality $\mu_{B}^*\leqslant\mu_A^*$ holds. The converse is true if $\mu$ satisfies all conditions of Propositions \ref{Pro: triangle 1}  and \ref{Pro: triangle 2}.
\end{prop}
\begin{proof}Since 
\[\mu_{\max}(\top):=\sup_{y\in\mathscr L\setminus\{\perp\}}\mu(\perp,y),\quad \mu_{\min}(\top):=\inf_{x\in\mathscr L\setminus\{\top\}}\mu(x,\top).\]
The equality $\mu_{\min}(\top)=\mu_{\max}(\top)$ is equivalent to 
\[\forall\,x\in\mathscr L\setminus\{\top\},\;\forall\,y\in\mathscr L\setminus\{\perp\},\quad \mu(\perp,y)\leqslant\mu(x,\top).\]
Joint with the inequalities $\mu_{\min}(y)\leqslant\mu(\perp,y)$ and $\mu(x,\top)\leqslant\mu_{\max}(x,\top)$, by 
Remark~\ref{Rem: mu A mu B star reinterpretation}, we obtain $\mu_B^*\leqslant\mu_A^*$.

In the case where $\mu$ satisfies the conditions of Propositions \ref{Pro: triangle 1} and \ref{Pro: triangle 2}, then one has \[\mu_{\min}(\top)=\mu_A^*\leqslant \mu_B^*=\mu_{\max}(\top).\]
Hence, if $\mu_B^* \leqslant \mu_A^*$, then the equality $\mu_{\min}(\top)=\mu_{\max}(\top)$ holds to be true.
\end{proof}

\begin{prop}\phantomsection\label{Pro: mu = mu max semistability}
Suppose that, for any $x\in\mathscr L\setminus\{\perp,\top\}$, \[\text{either $\mu(\perp,x)\not\leqslant\mu(\perp,\top)$, or $\mu(\perp,\top)\leqslant \mu(x,\top)$.}\]
If $\mu_{\max}(\top)=\mu(\perp,\top)$, then $\mu_{\min}(\top)=\mu_{\max}(\top)$.
\end{prop}

\begin{proof}
Let $x$ be an element of $\mathscr L\setminus\{\perp,\top\}$. Since $\mu_{\max}(\top)=\mu(\perp,\top)$, one has $\mu(\perp,x)\leqslant\mu(\perp,\top)$, and hence, by the hypothesis of the proposition, $\mu(\perp,\top)\leqslant\mu(x,\top)$. Hence $\mu_{\min}(\top)=\mu(\perp,\top)$.
\end{proof}

\begin{rema}
When $\mu$ is slope-like, for any $x\in\mathscr L\setminus\{\perp,\top\}$ the condition \[\text{either $\mu(\perp,x)\not\leqslant\mu(\perp,\top)$, or $\mu(\perp,\top)\leqslant \mu(x,\top)$}\] can be deduced from the point (3) of Definition \ref{Def: etoile gras} applied to $\operatorname{\perp} \leqslant x \leqslant \top$.
\end{rema}

\begin{prop}\phantomsection\label{Pro: mu = mu max semistability1}
Suppose that, for any $x\in\mathscr L\setminus\{\perp,\top\}$,
\[\text{either $\mu(\perp,x)\leqslant\mu(\perp,\top)$, or $\mu(\perp,\top)\not\leqslant \mu(x,\top)$.}\]
If $\mu_{\min}(\top)=\mu(\perp,\top)$, then $\mu_{\max}(\top)=\mu_{\min}(\top)$.
\end{prop}

\begin{proof}
The statement follows from Proposition \ref{Pro: mu = mu max semistability} by passing to the dual Harder-Narasimhan game.
\end{proof}

\begin{rema}
When $\mu$ is slope-like, for any $x\in\mathscr L\setminus\{\perp,\top\}$ the condition \[\text{either $\mu(\perp,x)\leqslant\mu(\perp,\top)$, or $\mu(\perp,\top)\not\leqslant \mu(x,\top)$.}\] can be deduced from the point (1) of Definition \ref{Def: etoile gras} applied to $\operatorname{\perp} \leqslant x \leqslant \top$.
\end{rema}

\begin{prop}\label{Pro: semistable and equilibrium}
Suppose that the pay-off function $\mu$ is slope-like. 
\begin{enumerate}[label=\rm(\arabic*)]
\item The following statements are equivalent:
\begin{enumerate}[label=\rm(\alph*)]
\item $\mu_{\max}(\top)=\mu(\perp,\top)$,
\item $\mu_{\min}(\top)=\mu(\perp,\top)$, 
\item $\mu_{\min}(\top)=\mu_{\max}(\top)$.
\end{enumerate}
\item If in addition $(\mathscr L,\leqslant)$ satisfies the conditions \ref{Item: condition 1} and \ref{Item: tilde 1} of Propositions \textup{\ref{Pro: triangle 1}} and \textup{\ref{Pro: triangle 2}}, then these conditions are also equivalent to 
\begin{enumerate}[label=\rm(\alph*)]
\setcounter{enumii}{3}
\item the Harder-Narasimhan game of pay-off function $\mu$ has Nash equilibrium.
\end{enumerate}
\end{enumerate}
\end{prop}
\begin{proof}
The implications (a) $\Longrightarrow$ (c) and (b) $\Longrightarrow$ (c) are consequences respectively of Propositions \ref{Pro: mu = mu max semistability} and \ref{Pro: mu = mu max semistability1}. The reverse implications (c) $\Longrightarrow$ (b) and (c) $\Longrightarrow$ (a) come from Proposition \ref{Pro: equivalent star} because $\mu$ is slope-like. The equivalence to (d) under the conditions \ref{Item: condition 1} and \ref{Item: tilde 1} of Propositions \textup{\ref{Pro: triangle 1}} and \textup{\ref{Pro: triangle 2}} is a consequence of Proposition \ref{Pro: semistable implies equilibrium1}. 
\end{proof}

\begin{exem}
Consider the Harder-Narasimhan game described in \S\ref{Sec: wdg} and illustrated by the following edge-weighted oriented graph.
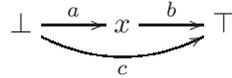
\begin{equation*}\xymatrix{\relax\perp\ar@/_1pc/[rr]_-{c}\ar[r]^-a&x\ar[r]^-b&\top}\end{equation*}
We assume that $(S,\leqslant)$ is totally ordered. As explained in \S\ref{Sec: wdg}, the Harder-Narasimhan game is semi-stable if and only if $a\leqslant b$. 
Note that
\begin{enumerate}[label=\rm(\arabic*)] 
\item $\mu_{\min}(\top)=b\wedge c$ is equal to $\mu(\perp,\top)=c$ if and only if $c\leqslant b$, 
\item $\mu_{\max}(\top)={a}\vee c$ is equal to $\mu(\perp,\top)=c$ if and only if $a\leqslant c$,
\item $\mu_{\min}(\top)= b\wedge c$ is equal to $ \mu_{\max}(\top) = a\vee c$ if and only if $a\leqslant c\leqslant b$. 
\item the inequality $\mu_A^*\leqslant\mu_B^*$ always holds\footnote{In fact, if $a\leqslant c$, then
$\mu_A^*=c\wedge b\leqslant \mu_B^*=a\vee(b\wedge c)$.
If $c\leqslant a$, then $b\wedge c\leqslant a$ and hence
$\mu_A^*=a\wedge b\leqslant a=\mu_B^*$.
}, and the Harder-Narasimhan game has Nash equilibrium if and only if $a\leqslant b$\footnote{ Clearly the inequality $\mu_B^*\leqslant\mu_A^*$ implies $a\leqslant b$. Conversely, if $a\leqslant b$, then $a\leqslant a\vee c$, $b\wedge c\leqslant b$ and $b\wedge c\leqslant a\vee c$. Hence the inequality $\mu_B^*\leqslant\mu_A^*$ holds.
}. 
\end{enumerate} 
\end{exem}

\begin{prop}\label{Pro: mu B mu A and equilibrium}
Assume that $(S,\leqslant)$ is a totally ordered set. If the Harder-Narasimhan game is semi-stable, then one has $\mu_B^*\leqslant\mu_A^*$. In particular, in the case where the conditions of Proposition \textup{\ref{Pro: triangle 1}} are satisfied or the conditions of Proposition \textup{\ref{Pro: triangle 2}} are satisfied, the Harder-Narasimhan game has Nash equilibrium.
\end{prop}
\begin{proof}
For any $x\in\mathscr L\setminus\{\perp\}$, one has
\[\mu_A(x)=\inf_{\begin{subarray}{c}
a\in\mathscr L\\
a<x
\end{subarray}}\sup_{\begin{subarray}{c}
b\in\mathscr L\\
a<b\leqslant x
\end{subarray}}\mu(a,b)\geqslant\inf_{\begin{subarray}{c}
a\in\mathscr L\\
a<x
\end{subarray}}\mu(a,x)=\mu_{\min}(x).\]
Taking the supremum with respect to $x$ one therefore has
\[\mu_B^*=\sup_{x\in\mathscr L\setminus\{\perp\}}\mu_{\min}(x)\leqslant\sup_{x\in\mathscr L\setminus\{\perp\}}\mu_A(x).\]
If the Harder-Narasimhan game is semi-stable, then 
\[\sup_{x\in\mathscr L\setminus\{\perp\}}\mu_A(x)\leqslant \mu_A^*.\]
The proposition is thus proved.
\end{proof}

The following simple example shows that even when the pay-off function of the Harder-Narasimhan game has values in a totally ordered set $(S, \leqslant)$, the semi-stability of the game is not equivalent to have the inequality $\mu_A^*\leqslant\mu_B^*$.In other words the inequality $\mu_A^*\leqslant\mu_B^*$ may fail without presuming anything about the semi-stability of the game.
\begin{exem}
Consider the following weighted and directed graph 
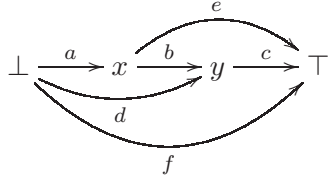
\[\xymatrix{\relax \perp\ar[r]^-a\ar@/_1pc/[rr]_-{d}\ar@/_2.5pc/[rrr]_-{f}&x\ar[r]^-{b}\ar@/^1.5pc/[rr]^-e&y\ar[r]^-{c}&\top}\]
where $a$, $b$, $c$, $d$, $e$ and $f$ take value in a totally ordered set $(S,\leqslant)$. 
By definition, one has 
\begin{gather*}\mu_A^*=(a\vee d\vee f)\wedge (b\vee e)\wedge c,\quad
\mu_B^*=a\vee(b\wedge d)\vee(c\wedge e\wedge f).\end{gather*}
In the case where $a$, $b$, $c$, $d$, $e$ and $f$ satisfy
\begin{equation}\label{Equ: mu B smaller than mu A}b<a<e<c,\quad f<a<d,\end{equation}
one has 
\[\mu_A^*=d\wedge e\wedge c=d\wedge e,\quad \mu_B^*=a\vee b\vee f=a.\]
Hence $\mu_B^*<\mu_A^*$.

Let us consider the semi-stability of this Harder-Narasimhan game. By definition, 
\begin{gather*}\mu_A(\top)=(a\vee d\vee f)\wedge(b\vee e)\wedge c,\\
\mu_A(y)=(a\vee d)\wedge b,\quad \mu_A(x)=a.
\end{gather*}
Under the assumption \eqref{Equ: mu B smaller than mu A}, one has 
\[\mu_A(\top)=d\wedge e,\quad \mu_A(y)=b<\mu_A(\top),\quad \mu_A(x)=a<\mu_A(\top).\] Hence the Harder-Narasimhan game is semi-stable.
\end{exem} 

\begin{prop}\label{Pro: equilibrium implies semistable}
Assume that, for any $x\in\mathscr L\setminus\{\perp\}$, the function $\mu_{[\perp,x]}$ satisfies the conditions of Proposition \textup{\ref{Pro: triangle 1}}. If the Harder-Narasimhan game has Nash equilibrium, then it is semi-stable.
\end{prop}
\begin{proof}
By Proposition \ref{Pro: triangle 1}, one has 
\[\forall\,x\in\mathscr L\setminus\{\perp\},\quad \mu_A(x)=\mu_{\min}(x).\]
Taking the supremum with respect to $x$, we obtain
\[\sup_{x\in\mathscr L\setminus\{\perp\}}\mu_A(x)=\mu_B^*.\]
Therefore, if $\mu_B^*=\mu_A^*=\mu_A(\top)$, then, for any $x\in \mathscr L\setminus\{\perp\}$, one has $\mu_A(x)\leqslant\mu_A(\top)$. Hence the Harder-Narasimhan game is semi-stable.
\end{proof}

If we combine the statements of Propositions \ref{Pro: semistable and equilibrium}, \ref{Pro: mu B mu A and equilibrium} and \ref{Pro: equilibrium implies semistable}, we obtain the following result. 
\begin{theo}\label{Thm: equivalence of stability conditions}
Assume that the following conditions are satisfied:
\begin{enumerate}[label=\rm(\roman*)]
\item  for any ascending chain 
$x_0< x_1<\ldots< x_n< x_{n+1}<\ldots$
of elements of $\mathscr L$, 
there exists $N\in\mathbb N$ such that  
$\mu(x_N,x_{N+1})\leqslant\mu(x_{N},\top)$ \textup{(}condition \ref{Item: condition 1} of Proposition \ref{Pro: triangle 1}\textup{)};
\item 
for any descending chain 
$x_0> x_1>\ldots> x_n>x_{n+1}>\ldots$
of elements of $\mathscr L$, there exists $N\in\mathbb N$ such that
$\mu(\perp,x_N)\leqslant\mu(x_{N+1},x_N)$ \textup{(}condition \ref{Item: tilde 1} of Proposition \ref{Pro: triangle 2}\textup{)};
\item the complete lattice is totally ordered and the pay-off function $\mu$ is slope-like \textup{(}in particular, the conditions \ref{Item: condition 2} and \ref{Item: tilde 2} of Propositions \ref{Pro: triangle 1} and \ref{Pro: triangle 2} are satisfied by $\mu$ and also by its restrictions\textup{)}.
\end{enumerate}
Then the following statements are equivalent:
\begin{enumerate}[label=\rm(\alph*)]
\item $\mu_{\max}(\top)=\mu(\perp,\top)$;
\item $\mu_{\min}(\top)=\mu(\perp,\top)$;
\item $\mu_{\min}(\top)=\mu_{\max}(\top)$;
\item the Harder-Narasimhan game of pay-off function $\mu$ has Nash equilibrium;
\item the Harder-Narasimhan game of pay-off function $\mu$ is semi-stable.
\end{enumerate}
\end{theo}

\begin{exem}\label{Ex: classical HN game}
Consider the Harder-Narasimhan game associated with a non-zero vector bundle $E$ on a regular projective curve $C$, as described in the beginning of the section. Note that the pay-off function of this Harder-Narasimhan game is slope-like, as shown by Proposition \ref{Pro: degree function}. Moreover, the bounded lattice $\mathscr L_E$ satisfies the ascending chain condition. Therefore, the conditions \ref{Item: condition 1} and \ref{Item: condition 2} of Proposition \ref{Pro: triangle 1} are satisfied. In particular, one has $\mu_A(E)=\mu_{\min}(E)$.
By Remark \ref{Rem: strong descending chain condition}, the condition \ref{Item: tilde 1} of Proposition \ref{Pro: triangle 2} is also satisfied. Therefore, by Theorem \ref{Thm: equivalence of stability conditions}, we obtain that the following statements are equivalent:
\begin{enumerate}[label=\rm(\alph*)]
\item The vector bundle $E$ is semi-stable in the usual sense,  namely \[\mu(E)=\mu_{\max}(E):=\sup_{0\neq F\subseteq E}\mu(F);\]
\item $\mu(E)=\mu_{\min}(E)$;
\item $\mu_{\max}(E)=\mu_{\min}$(E);
\item The Harder-Narasimhan game has Nash equilibrium;
\item The Harder-Narasimhan game is semi-stable, namely, for any non-zero vector subbundle $F$ of $E$, one has
$\mu_{\min}(F)\leqslant\mu_{\min}(E)$.
\end{enumerate}

\end{exem}

\begin{exem}

Let $\mathbb R[T]$ be the real vector space of polynomials of one variable $T$. We equip $\mathbb R[T]$ with the order $\leqslant$ defined as follows:
\[P\leqslant Q \Longleftrightarrow \exists\,N\in\mathbb N,\;\forall\,n\in\mathbb N_{\geqslant N},\; P(n)\leqslant Q(n). \]
Note that $P<Q$ if and only if the leading coefficient of the polynomial $Q-P$ is positive. Hence it is a total order on $\mathbb R[T]$. Moreover, $\mathbb R[T]$ is a totally ordered vector space over $\mathbb R$. We denote by $(S,\leqslant)$ the Dedekind-MacNeille completion of $(\mathbb R[T],\leqslant)$\textcolor{red}{.}

Let $k$ be a field, $X$ be a projective scheme over $\Spec k$ and $d$ be the Krull dimension of $X$. We fix an ample line bundle $\mathcal O_X(1)$ on $X$. If $E$ is a coherent $\mathcal O_X$-module, we denote by $H_E$ the Hilbert polynomial of $E$, which is a polynomial in $\mathbb R[T]$ such that 
\[\forall\,n\in\mathbb N_{\geqslant 1},\quad P_E(n)=\chi(E\otimes\mathcal O_X(n)).\]
We denote by $r_E$ the coefficient of $T^d$ in the polynomial $H_E$. Note that $r_E$ is always non-negative (it is positive precisely when the support of $E$ is of dimension $d$). If 
\[\xymatrix{0\ar[r]&E'\ar[r]&E\ar[r]&E''\ar[r]&0}\]
is a short exact sequence of coherent $\mathcal O_X$-modules, then the following equalities hold:
\[r_{E}=r_{E'}+r_{E''},\quad H_E=H_{E'}+H_{E''}.\]

We fix a non-zero coherent $\mathcal O_X$-module $E$ and denote by $\mathscr L_E$ the set of coherent sub-$\mathcal O_X$-module of $E$ equipped with the order of inclusion. Then $(\mathscr L_E,\subseteq)$ forms a bounded lattice. Moreover, the map
\[\mu:P_{\subseteq}(\mathscr L_E)\longrightarrow S,\quad \mu(F,E):=\begin{cases}
r_{E/F}^{-1}Q_{E/F},&r_{E/F}>0,\\
+\infty,&\text{else},
\end{cases}\]
defines a pay-off function of a Harder-Narasimhan game. By Proposition \ref{Pro: degree function}, the function $\mu$ is slope-like. By Theorem \ref{Thm: equivalence of stability conditions}, we obtain that the Nash equilibrium of this Harder-Narasimhan game is equivalent to the Gieseker semi-stability of the coherent sheaf $E$ with respect to the polarisation $\mathcal O_X(1)$ (see \cite[\S1.2]{MR2665168}).

\end{exem}

\begin{exem}
Let $G$ be a reductive group scheme over a regular projective curve $C$. Note that the pay-off function of the game defined in section \ref{Sec: red gps over curve} is not slope-like. One could therefore seek for a Harder-Narasimhan game with slope-like pay-off function that would better take into account the fact the underlying vector bundle setting (as all this discussion is possible because the Lie algebra of an algebraic group over a curve is in particular a vector bundle). One can therefore consider the Harder-Narasimhan game associated with the Lie algebra $\mathfrak{g}$ of $G$ on the regular projective curve $C$, as described in the beginning of the section. Note that if this game is semi-stable then by Theorem \ref{Thm: equivalence of stability conditions} for any vector subbundle $F \subseteq \mathfrak{g}$ one has $\mu(F) \leq \sup_{0\subsetneq F\subseteq \mathfrak{g}}\mu(F) = \mu(\mathfrak{g}) =0$. In particular, any parabolic subgroup of $G$ has negative degree, so if this slope-like game is semistable so is the game defined in Section \ref{Sec: red gps over curve}.

If now this slope-like game is not semi-stable, let 
\[0= E_0 \subseteq E_{1} \subseteq \cdots \subseteq E_{n-1} \subseteq \mathfrak{g}\]
be the Harder-Narasimhan filtration associated to the game. Then from the properties of the Harder-Narasimhan filtration together with the example \ref{Ex: classical HN game} one deduces that $\mu_A(\mathfrak{g}) \leqslant 0$, $\mu_A(E_1) \geqslant 0$. Indeed if $0 < \mu_A(\mathfrak{g})$ the slope $\mu(E_{n-1}, \mathfrak{g})$ is strictly positive, so $\deg(E_{n-1})$ is strictly negative. Moreover, as $\mu_A(E_{i-1}, E_{i})= \mu(E_{i-1}, E_{i}) > \mu_A(E_{i}, E_{i+1}) =  \mu(E_{i}, E_{i+1})$ for any $i \in \{1, \ldots, n-1\}$ the slopes $\mu(E_{i}, E_{i+1})$ are strictly positive, therefore the $\deg(E_i)$'s are strictly negative for $i \in \{1, \ldots, n-1\}$. But then (a) tells us that any subbundle of $\mathfrak{g}$ is of negative degree, so $\mathfrak{g}$ is semistable, hence a contradiction. On the same spirit, if $\mu_A(E_1) < 0$ then $\deg(E_1) <0$, for any $i \in \{1, \ldots, n\}$ one has $\deg(E_i) < \deg(E_{i-1})$, leading to a contradiction as $\deg(\mathfrak{g}) = 0$.
%

\end{exem}

\subsection{Jordan-Hölder filtration}

In this subsection, we consider a Harder-Narasimhan game which is known to be semi-stable and we study its stable filtrations. Let $(\mathscr L,\leqslant)$ be a bounded lattice and $\mu$ be a mapping from $P_{<}(\mathscr L)$ to a complete lattice $(S,\leqslant)$, which defines a Harder-Narasimhan game. We assume the following:
\begin{enumerate}[label=\rm(\roman*)]
\item the bounded lattice $(\mathscr L,\leqslant)$ satisfies the ascending chain condition (in particular, the condition \ref{Item: condition 1} of Proposition \ref{Pro: triangle 1} is satisfied);
\item\label{Item: stronger decreasing condition} for any descending chain
\[x_0>x_1>\ldots>x_n>x_{n+1}>\ldots\]
of elements of $\mathscr L$, there exists $N\in\mathbb N$ such that $\mu(x_{N+1},x_N)=+\infty$ (in particular, condition \ref{Item: tilde 1} of Proposition \ref{Pro: triangle 2} is satisfied);
\item $(S,\leqslant)$ is a totally ordered set and the function $\mu$ is slope-like;
\item the Harder-Narasimhan game is semi-stable, namely (see Theorem \ref{Thm: equivalence of stability conditions})
\[\forall\,x\in\mathscr L\setminus\{\perp\},\quad \mu(\perp,x)\leqslant\mu(\perp,\top).\]
\end{enumerate}

\begin{theo}
Assume that $\mu(\perp,\top)\neq+\infty$. There exists a sequence  
\begin{equation}\label{Equ:JH filtration}\operatorname{\top}=y_0>y_1>\ldots>y_n=\operatorname{\perp}\end{equation}
such that, for any $i\in\{1,\ldots,n\}$, $\mu(y_{i},y_{i-1})=\mu(\perp,\top)$ and
\[\forall\,z\in\mathscr L\text{ such that $y_{i}<z<y_{i-1}$},\; \mu(y_{i},z)<\mu(y_{i},y_{i-1}).\]
\end{theo}
\begin{proof}
If for any $x\in\mathscr L\setminus\{\perp,\top\}$ one has $\mu(\perp,x)<\mu(\perp,\top)$, the choice of $n=1$ and $(y_1,y_0)=(\perp,\top)$ satisfies the required condition. Otherwise the set 
\[\big\{x\in\mathscr L\setminus\{\perp,\top\}\,:\,\mu(\perp,x)=\mu(\perp,\top)\big\}\]
is not empty and by the ascending chain condition we could pick a maximal element of this set and let $y_1$ be this element. By Proposition \ref{Pro: equivalent star}, one has $\mu(y_1,\top)=\mu(\perp,\top)$. Moreover, for any $x\in\mathscr L$ such that $y_1<x<\top$, one has $\mu(\perp,x)<\mu(\perp,\top)=\mu(\perp,y_1)$. Hence, by Proposition \ref{Pro: equivalent star} one has 
\[\mu(y_1,x)<\mu(\perp,x)<\mu(\perp,y_1)=\mu(y_1,\top).\]
Iterating this procedure we obtain a decreasing sequence
\[\top=y_0>y_1>\ldots\]
such that $\mu(y_i,y_{i-1})=\mu(\perp,\top)$ and 
\[\forall\,z\in\mathscr L\text{ such that $y_{i}<z<y_{i-1}$},\; \mu(y_{i},z)<\mu(y_{i},y_{i-1})\]
for any $i$. This procedure terminates within finitely many steps since otherwise by the condition \ref{Item: stronger decreasing condition} there would exist $N\in\mathbb N$ such that $\mu(y_{N+1},y_{N})=+\infty$, which contradicts the hypothesis $\mu(\perp,\top)<+\infty$. 
\end{proof}

\begin{rema}
The sequence \eqref{Equ:JH filtration} in the above theorem is called a Jordan-Hölder filtration of the Harder-Narasimhan game. It is usually not unique. In the case where the pay-off function $\mu$ is affine, namely for any $(a,b)\in\mathscr L^2$ such that $a\not\leqslant b$ one has
\[\mu(a\wedge b,a)=\mu(b,a\vee b),\] it can be shown that all Jordan-Hölder filtrations have the same length. Suppose that 
\[\top=x_0>x_1>\ldots>x_m=\operatorname{\perp}\]
is another Jordan-Hölder filtration of the Harder-Narasimhan game (with an affine pay-off function). Let $i\in\{1,\ldots,n\}$ be the largest index such that $x_{m-1}\leqslant y_{i-1}$. Since $x_{m-1}\not\leqslant y_i$, one has 
\begin{equation}\label{Equ:jh}\mu(\perp,\top)=\mu(\perp,x_{m-1})\leqslant\mu(x_{m-1}\wedge y_i,x_{m-1})=\mu(y_i,x_{m-1}\vee y_i).\end{equation} 
Since $y_i<x_{m-1}\vee y_i\leqslant y_{i-1}$, we obtain that $x_{m-1}\vee y_i=y_{i-1}$ since otherwise \[\mu(y_i,x_{m-1}\vee y_i)<\mu(y_{i},y_{i-1})=\mu(\perp,\top),\] which leads to a contradiction.

For any $j\in\{1,\ldots,n\}$, by Proposition \ref{Pro: mu A star inequality} one has 
\[\mu(\perp,\top)\geqslant\mu(\perp,y_j\vee x_{m-1})\geqslant\mu_{\min}(\perp, y_j\vee x_{m-1})\geqslant\mu(\perp,\top).\]
Hence $\mu(\perp,y_j\vee x_{m-1})=\mu(\perp,\top)$. By  Proposition \ref{Pro: equivalent star} we obtain
\[\forall\,j\in\{1,\ldots,n\},\quad\mu(y_{j}\vee x_{m-1},y_{j-1}\vee x_{m-1})=\mu(\perp,\top).\]
If $y_{j}\vee x_{m-1}<y_{j-1}\vee x_{m-1}$, then for any $w\in\mathscr L$ such that
\[y_j\vee x_{m-1}<w<y_{j-1}\vee x_{m-1},\]
one has $y_{i-1}\not\leqslant w$. Since the pay off function is affine, the equality
\[\mu(y_{i-1}\wedge w,y_{i-1})=\mu(w,y_{i-1}\vee w)=\mu(w,y_{i-1}\vee x_{m-1})\]
is satisfied.
Since $y_i\leqslant y_{i-1}\wedge w<y_{i-1}$,  by Proposition \ref{Pro: equivalent star} and the condition 
\[\mu(y_i,y_{i-1}\wedge w)<\mu(y_i,y_{i-1})=\mu(\perp,\top)\]
we obtain 
\[\mu(y_{i-1}\wedge w,y_{i-1})>\mu(y_{i},y_{i-1})=\mu(\perp,\top).\]
Still by Proposition \ref{Pro: equivalent star} we deduce 
\[\mu(y_{i}\vee x_{m-1},w)<\mu(\perp,\top).\]
Therefore, from the sequence
\[\top=x_{m-1}\vee y_0\geqslant\ldots\geqslant x_{m-1}\vee y_n=x_{m-1}\]
we could extract a Jordan-Hölder filtration of the restriction of the Harder-Narasimhan game to $\mathscr L_{[x_{m-1},\top]}$, whose length is $\leqslant n-1$ since $x_{m-1}\vee y_i=y_{i-1}=x_{m-1}\vee y_{i-1}$. Iterating this procedure we obtain $m\leqslant n$. By the symmetry between the two Jordan-Hölder filtrations we get $m=n$.
\end{rema}

\begin{exem}
Back to the example derived in Section \ref{sec : coprimary filtration}, let $R$ be a N\oe therian ring, and let $M$ be a $\mathfrak{p}$-coprimary $R$-module, where $\mathfrak{p}$ is a prime ideal of $R$. A Jordan--Hölder filtration of $M$ is a filtration whose successive quotients are isomorphic to $R/\mathfrak{p}$.
\end{exem}

\backmatter
\bibliography{HNgame}
\bibliographystyle{smfplain}

\end{document}